\documentclass[12pt]{amsart}
\usepackage{amsmath, amssymb, amsthm, amsfonts, mathrsfs}
\usepackage{url}
\usepackage{cite}
\usepackage{amscd}
\usepackage{graphicx}
\usepackage{caption}
\usepackage{subcaption}
\usepackage{fullpage}
\usepackage{lscape, pdflscape}

\newtheorem*{conjecture}{Conjecture}
\newtheorem{prop}{Proposition}

\newtheorem{theorem}[prop]{Theorem}

\theoremstyle{definition}

\newcommand{\ds}[1]{\ensuremath{\displaystyle{#1}}}
\newcommand{\incentive}{\varphi}

\begin{document}
\title{The Inherent Randomness of Evolving Populations}
\author{Marc Harper}
\address{University of California Los Angeles}
\email{marcharper@ucla.edu} 
% \authorinfo{\texttt{c}}
\date{\today}
% \subjclass[2000]{Primary: 37N25; Secondary: 91A22, 94A15}
% \keywords{evolutionary game theory, information geometry, natural selection, Moran process, Bayesian inference}

\begin{abstract}
The entropy rates of the Wright-Fisher process, the Moran process, and generalizations are computed and used to compare these processes and their dependence on standard evolutionary parameters. Entropy rates are measures of the variation dependent on both short-run and long-run behavior, and allow the relationships between mutation, selection, and population size to be examined. Bounds for the entropy rate are given for the Moran process (independent of population size) and for the Wright-Fisher process (bounded for fixed population size). A generational Moran process is also presented for comparison to the Wright-Fisher Process. Results include analytic results and computational extensions.
\end{abstract}

\maketitle

\section{Introduction}

Populations of replicating entities are subject to a variety of selective, stochastic, and diversifying processes. These processes can act on different time scales, ranging from short-term stochastic drift in small populations to long-term selective processes that slowly lead to the fixation of one trait in a population. Many such processes have been studied extensively, including founder and bottleneck effects, selective stability, and mutation-selection balances \cite{moran1958random} \cite{kimura1985neutral} and depend significantly on certain population parameters, such as population size \cite{fogel1998instability}  \cite{ficici2000effects} \cite{nowak2004emergence} \cite{patwa2008fixation}. Accordingly, variation in population dynamics arises from both short-term and long-term effects.

Why study the entropy rate of models of evolutionary processes? Entropy rates measure the inherent randomness of a process due to both short-term and long-term dynamics. Moreover, by assigning a value to each particular process we gain the ability to not only compare processes but also a way to explore the interactions of the fundamental processes of population biology: natural selection (through the fitness landscape), genetic drift (via the population size), and diversifying processes such as mutation. We can also study the fundamental processes in evolutionary dynamics pairwise by considering several limits: to eliminate drift, we can let the population size $N \to \infty$, to eliminate mutation, we can let the mutation probability $\mu \to 0$, and to eliminate selection, we can use a uniform fitness landscape. Entropy rates reveal that there is significant long-run variation in finite population dynamics even in cases that are thought of as \emph{evolutionarily stable}, and also situations in which the ``most inherently random'' behavior occurs for relatively large populations rather than from the stochastic effects of small populations.

\subsection{The Moran Process and Generalizations}

The Moran process is a birth-death process that describes natural selection in finite populations \cite{moran1962statistical} and has many applications \cite{dingli2011stochastic} \cite{traulsen2009stochastic}. In each round of the process, an individual is chosen proportionally to fitness to reproduce and an individual is chosen at random to be replaced. The classical Moran Process was generalized to include mutation and frequency dependent fitness by Fudenberg et al \cite{fudenberg2004stochastic}. Let us consider a slight generalization to include possibly variable mutation rates that depend on the population state. For a population of size $N$, let the population be divided into two types $A$ and $b$, with the number of $A$ individuals denoted by $i$ and the number of $B$ individuals by $N-i$. A pair $(i, N-i)$ with $0 \leq i \leq N$ is a population state. Let $f_A$ and $f_B$ be the fitness of the types $A$ and $B$ respectively, possibly depending on the population state (i.e. is frequency-dependent). The Moran process has transition probabilities
\begin{align}\label{moran_process}
T_{i \to i+1} &= \frac{i f_A(i)(1 - \mu_{AB}(i)) + (N-i) f_B(i)\mu_{BA}(N-i)}{i f_A(i) + (N-i) f_B(i)} \frac{N-i}{N} \notag \\
T_{i \to i-1} &= \frac{i f_A(i)\mu_{AB}(i) + (N-i) f_B(i)(1 - \mu_{BA}(N-i))}{i f_A(i) + (N-i) f_B(i)} \frac{i}{N}, \\
T_{i \to i} &= 1 - T_{i \to i+1} - T_{i \to i-1}, \notag
\end{align}
where $\mu_{AB}$ and $\mu_{BA}$ are mutation probabilities that may depend on the state, and the fitness landscape is given by
\begin{align*}
f_A(i) &= \frac{a(i-1) + b(N-i)}{N-1} \\
f_B(i) &= \frac{ci + d(N-i-1)}{N-1}
\end{align*}
for a game matrix defined by
\[ \left( \begin{matrix}
 a & b\\
 c & d
\end{matrix} \right) \]
\noindent In accordance with \cite{claussen2005non} and \cite{fudenberg2004stochastic}, further assume that $T_{0 \to 1} = \mu_{AB}$, $T_{0 \to 0} = 1- \mu_{AB}$, $T_{N \to N-1} = \mu_{BA}$, and $T_{N \to N} = 1 - \mu_{BA}$ so that the Markov process has a stationary distribution and no absorbing states. We will consider two mutation regimes: the \emph{boundary} regime defined by $\mu_{AB}(i) = 0 = \mu_{BA}(i)$ for $i \neq 0, N$ and the \emph{uniform} regime defined by $\mu_{AB}(i) = \mu_{AB}$ and $\mu_{BA}(i) = \mu_{BA}$ for all $i$ (so that the mutation rates are constant). The uniform regime is a more realistic model of mutation whereas the boundary regime is the minimal amount of mutation required to ensure a stationary distribution for the Moran process in most of the cases we will consider. If $i f_A(i) = (N-i) f_B(i)$ for all $i \neq 0, N$, then the two regimes are equivalent. 

\subsection{The Wright-Fisher Process}

In contrast to the Moran process, which models a population in terms of individual birth-death events, the Wright-Fisher process is a generational model of evolution \cite{imhof2006evolutionary} \cite{ewens2004mathematical}. Each successive generation is formed by sampling, proportionally to fitness, the current generation. Define the Wright-Fisher Process with mutation for evolutionary games by the following transition probabilities:
\begin{align*}
 T_{i \to j} = \binom{N}{j} &\left(\frac{i f_A(i)(1 - \mu_{AB}(i)) + (N-i) f_B(i)\mu_{BA}(N-i)}{i f_A(i) + (N-i) f_B(i)}\right)^j \\ \times &\left(\frac{i f_A(i)\mu_{AB}(i) + (N-i) f_B(i)(1 - \mu_{BA}(N-i))}{i f_A(i) + (N-i) f_B(i)}\right)^{N-j}
 \label{wright_fisher_process}
\end{align*}
This is a slight generalization of the basic process as given by Imhof and Nowak \cite{imhof2006evolutionary} to include mutation, though we will not consider parameters for intensity of selection. In contrast to the Moran process, the Wright-Fisher process is not tridiagonal, rather every state is accessible from every other state, so long as the fitness landscape is non-zero. For more on both the Moran process and the Wright-Fisher process see \cite{nowak2006evolutionary}.

\subsection{Entropy Rate}

A fundamental tool in information theory, probability, and statistics is the Shannon entropy of a probability distribution \cite{shannon1949mathematical}. For a discrete probability distribution $p = (p_0, p_1, \ldots, p_n)$, the Shannon entropy (or simply entropy) is
\[ H(p) = -\sum_{i=0}^{n}{ p_i \log{p_i}},\]
where $p_i \log{p_i} = 0$ if $p_i = 0$. The meaning of the entropy of a probability distribution is often described as a measure of uncertainty or information content. The entropy rate of a stationary Markov process is an information-theoretic quantity that characterizes the inherent randomness of the process \cite{cover2006elements} \cite{strelioff2007inferring}, and plays a similar role as the Shannon entropy. To each state $i$ of a Markov process $P$ there is a probability distribution $T_i = (T_{i \to 0}, \ldots, T_{i \to n})$ for the transition probabilities out of the state. We refer to the entropies of these transition probability distributions as the transition entropies $H(T_i)$. The mean of the transition entropies taken with respect to the stationary distribution $s=(s_0, \ldots, s_n)$ of the Markov process is the entropy rate:
\begin{align*}
\label{entropy_rate}
H(P) &= \sum_{i=0}^{n}{s_i \sum_{j=0}^{n}{T_{i \to j} \log T_{i \to j}} } \\
     &= \sum_{i=0}^{n}{s_i H(T_i) }
\end{align*}

The stationary distribution is a description of the long term behavior of a Markov process, and so the entropy can be similarly interpreted as a measure of the uncertainty, inherent randomness, or information content of the long run behavior of the process. The entropy rate is affected by the likelihood that the process occupies a particular state and the entropy of the behavior of the state. In other words, the entropy rate reflects both the long term variance in population states (the stationary distribution) and the short term variance due to the entropy of the transition probabilities at the states represented significantly in the stationary distribution.

Generally for a Markov process the transition probabilities are known a priori; the stationary distribution, however, can be difficult to describe analytically, depending on the complexity of the transition probabilities. Since the maximal Shannon entropy for a discrete distribution on $n$ states is $\log n$, the theoretical maximum entropy rate for a Markov process is also $\log n$. For a tridiagonal process (e.g. the Moran process), the maximum entropy rate is $\log 3$. For the Wright-Fisher process, the theoretical maximum entropy is $\log{(N+1)}$, where $N$ is the population size, because there are $N+1$ states (and so typically $N+1$ non-zero values in each transition distribution).

Stationary distributions for the Moran process and some recently-studied generalizations are given by Claussen and Traulsen in \cite{claussen2005non} (see also \cite{antal2009strategy}), the computation of which we briefly discuss. The components $s_i$ of the stationary distribution satisfy $s_i T_{i \to i+1} = s_{i+1} T_{i+1 \to i}$ and
\begin{equation}
s_j = s_0 \prod_{i=0}^{j-1}{\frac{T_{i \to i+1}}{T_{i+1 \to i}}},
\label{s_j}
\end{equation}
where $s_0$ can be obtained from the normalization $\sum_i{s_i} = 1$:
\begin{equation}
s_0 = \left(1 + \sum_{j=1}^{N}{\prod_{i=0}^{j-1}{\frac{T_{i \to i+1}}{T_{i+1 \to i}}}}\right)^{-1}
\label{s_0}
\end{equation}
This particular formulation relies on the fact that the processes are tridiagonal with only transitions between neighboring states being nonzero. From a computational perspective, for any concrete values of the various parameters of these processes, the stationary distribution can be computed efficiently even for relatively large populations using a sparse matrix approach, and useful analytic forms can be given in some cases. Finally, note that nonzero mutation probabilities on the boundary states $i=0, N$ are required so that the Markov process has a unique stationary distribution. In other words, we must prevent these states from being absorbing, and we can recover the behavior of processes without mutation by letting $\mu$ tend to zero. Analytic solutions for some examples of the Moran process on evolutionary games in the boundary regime are given in \cite{claussen2005non}. See Figure \ref{figure_0} for an example of a Moran process with associated transition entropies and the stationary distribution.

\begin{figure}
    \centering
    \includegraphics[width=0.9\textwidth]{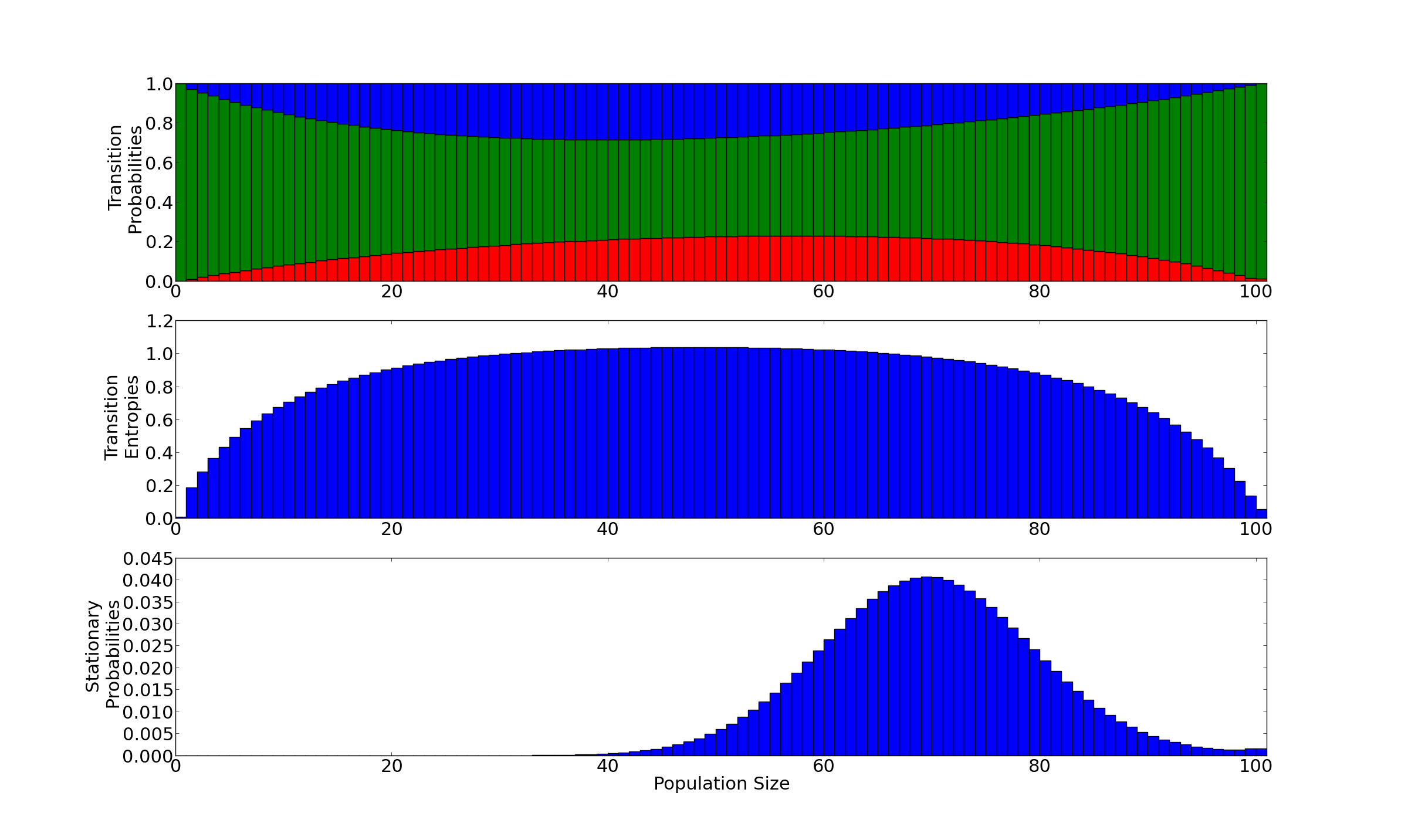}
    \caption{Top: Transition Probabilities for $N=100$, $\mu_{AB} = 0.001$, $\mu_{BA} = 0.01$, uniform mutations, and game matrix given by $a=2$, $b=4$, $c=3$, and $d=2$. Red, green, and blue correspond to $T_{i \to i-1}, T_{i \to i}, T_{i \to i+1}$ for each state $i$. Middle: Transition entropies. Bottom: stationary distribution. The entropy rate in this example is approximately $H = 0.9472$. Stationary distributions are not generally Gaussian \cite{claussen2005non}.}
    \label{figure_0}
\end{figure}

For the Moran process we will consider several fundamental examples and give a variety of analytical results. Calculation of the stationary state for the Wright-Fisher process is not as easy, computationally or analytically, though see \cite{imhof2006evolutionary} for some results. We will give computational results for comparison in some cases for the Wright-Fisher process, along with some analytical results and conjectures.

Finally, we will make use of one additional information-theoretic quantity called the Kullback-Leibler divergence \cite{kullback1951information} (or KL-divergence), which is a measure of ``distance'' between probability distributions:
\begin{equation}
 D_{KL}(p||q) = \sum_{i}{p_i \log{p_i} - p_i \log{q_i}}
\label{kl_divergence} 
\end{equation}
This divergence is not a true distance function in the sense of a metric (it does not satisfy the triangle inequality); nevertheless it is a widely used measure of difference between probability distributions. See \cite{cover2006elements} and \cite{shannon1949mathematical} for more on any of the mentioned information theory topics.

\subsection{$n$-fold Moran process}

Since the Wright-Fisher process is a generational process, replacing the entire population in each iteration, and the Moran process is atomic process, they exhibit very different behaviors. Consider the following process, which will be referred to as the $n$-fold Moran process. Define each step of the process as k steps of the Moran process, so that $n=1$ is the Moran process, and $n=N$, where $N$ is the population size, yields a \emph{generational} processes derived from the Moran process. 

The transition probabilities of the $n$-fold process can be computed directly from the transition matrix of the Moran process (equation \ref{moran_process}) by simply computing the $n$-th power of the transition matrix. Since the transition matrix of the Moran process is tridiagonal, each iterate will have two more nonzero diagonals corresponding to the two new population states accessible in each step of the compressed process. Moreover, since the stationary distribution of a Markov chain can be obtained by the rows of the matrix defined by
\[ s = \lim_{m \to \infty}{T^m} = \lim_{m \to \infty}{\left(T^k\right)^m}, \]
the stationary distributions of the $n$-fold Moran process are the same for all $k$, given a fixed transition matrix $T$. The entries of the transition matrix, $T^n_{a \to a'}$ correspond to the probability of moving from population state $(a,N-a)$ to population state $(a',N-a')$ in exactly $n$ steps of the Moran process.

\section{Results}

\subsection{Neutral Evolution: Moran Process}

First consider a population where both types have equal and constant fitness, i.e. for the game matrix of all ones, or more generally, when $f_A(i) = f_B(i)$ for all $i$. Figure \ref{figure_1} shows the entropy rate as a function of the population size $N$ for various $\mu$ for the Moran process. In the case $N=2$, it is easy to show that
\[ H(P) = \frac{2}{2+4 \mu} H((\mu, 1-\mu)) + \frac{4 \mu}{2 + 4\mu} \left(\frac{3}{2} \log 2 \right) ,\]
where $H((\mu, 1-\mu))$ is the binary entropy function (the Shannon entropy for the distribution $(\mu, 1-\mu)$). In this special case, the two mutation regimes yield the same process, which is typically not true for $N > 2$. Though this is a very simple case, it illustrates some common features of these processes. For instance, as $\mu \to 0$, the entropy rate $E \to 0$, a fact which holds for a wide variety of such processes. (We will also see that the constant $\frac{3}{2} \log 2$ plays a special role.)

Fudenberg et al. show in \cite{fudenberg2004stochastic} that if $k = \mu_{BA} / \mu_{AB}$ is fixed along with the population size $N$, then for the uniform mutation regime the stationary distribution of the process converges to
\[ s = \left(\frac{k \rho_A}{k \rho_A + \rho_B}, 0, \ldots, 0, \frac{\rho_B}{k\rho_A + \rho_B} \right),\]
where $\rho_A$ and $\rho_B$ are the fixation probabilities of the types $A$ and $B$ respectively when the population starts with a single individual of the type respectively. The fixation probability depends on the fitness landscape, which is not necessarily neutral. This is an essential ingredient for the following result (all proofs in appendix).

\begin{figure}
    \centering
    \includegraphics[width=0.45\textwidth]{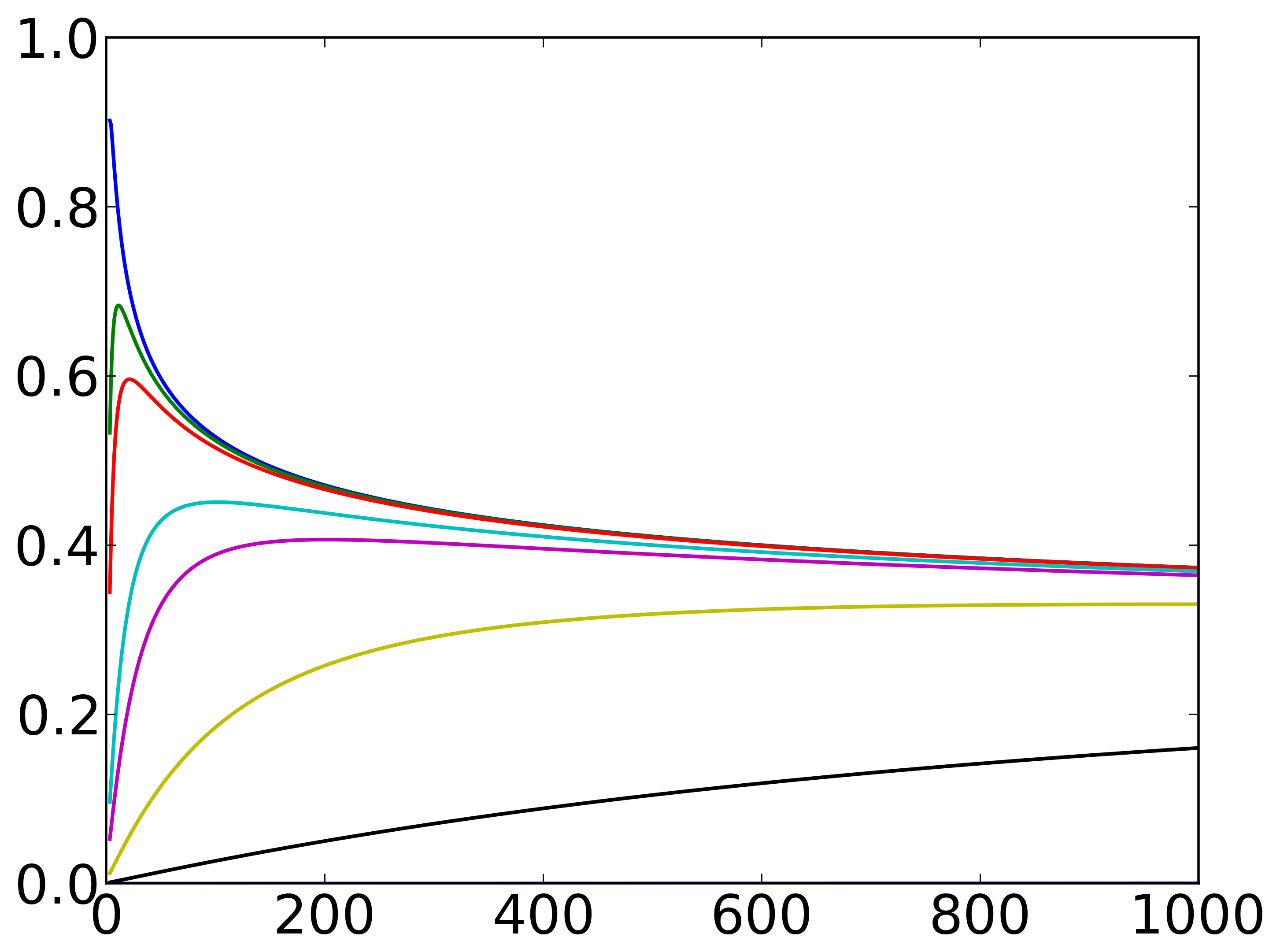}
    \includegraphics[width=0.45\textwidth]{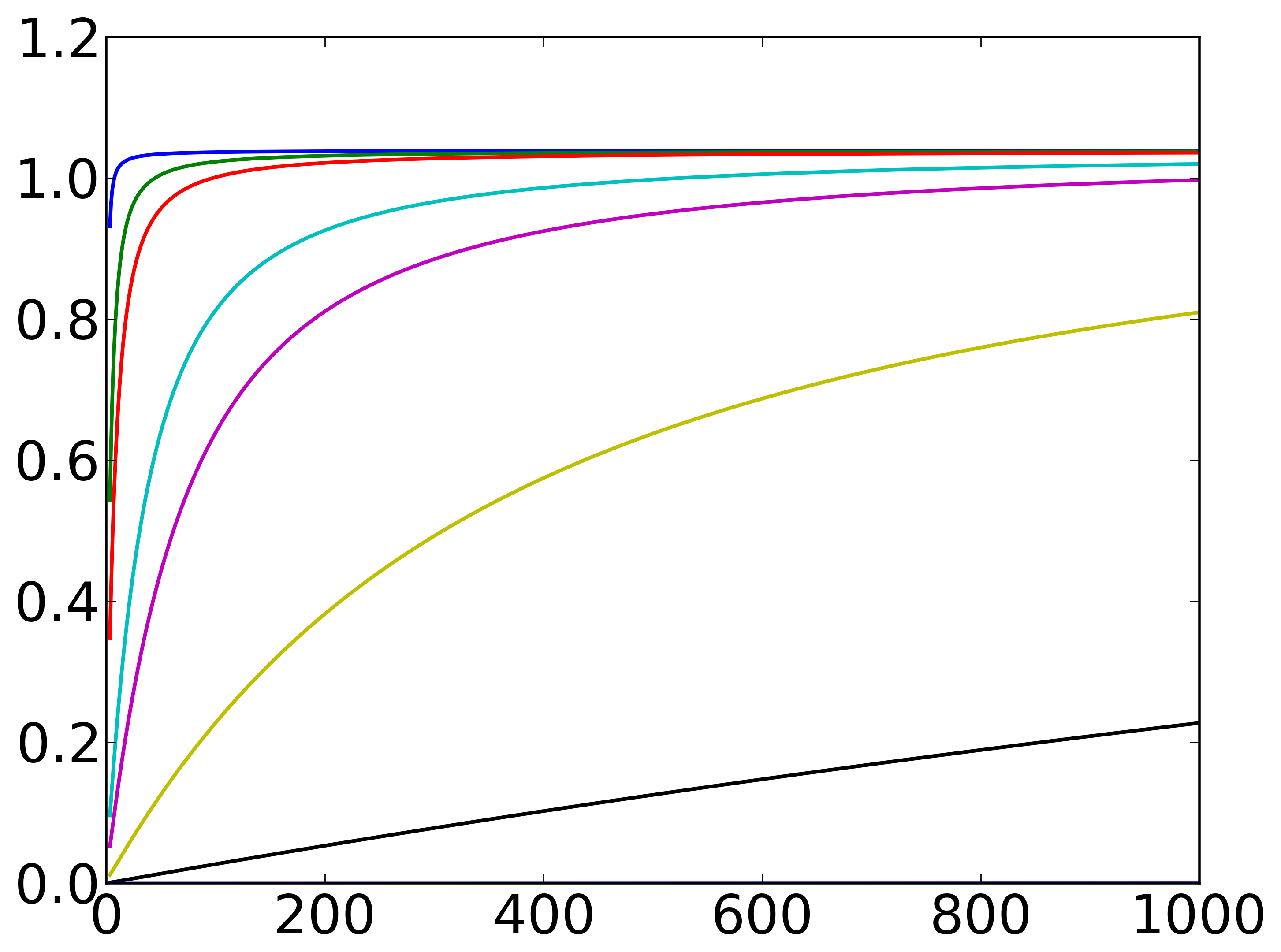}
    \caption{Entropy Rate vs. Population Size $N$ for $\mu_{AB} = \mu_{BA} \in \{0.5, 0.1, 0.05, 0.01, 0.005, 0.001, 0.0001\}$ (top to bottom) with a neutral fitness landscape. Left: Mutations only at the boundary states (boundary regime); the entropy rate eventually approaches zero as $N \to \infty$. Right: Mutations for all states (uniform regime); the entropy rate approaches $3/2 \log 2$ as $N \to \infty$.}
    \label{figure_1}
\end{figure}

\begin{theorem}
Let $\mu_{AB}=\mu$ and $\mu_{BA} = k \mu$. For the Moran process (Equations \ref{moran_process}) with the uniform mutation regime and otherwise arbitrary parameters, $\lim_{\mu \to 0}{H(P)} = 0$.
\end{theorem}

Though simple to prove given the result of \cite{fudenberg2004stochastic}, this theorem embodies an important fact about mutation in evolutionary processes. In this case the entropy rate reflects the fact that in the absence of mutation, the long run behavior of the population is fixation on one of two types, and the inherent randomness of the population dynamics is eliminated. The same limit holds for the boundary regime as well.

\begin{theorem}
For the boundary mutation regime and assumptions otherwise the same as in Theorem 1, $\lim_{\mu \to 0}{H(P)} = 0$.
\end{theorem}

\subsection{Large Populations} It may be tempting intuitively to think that for large population sizes that the entropy rate also tends to zero since an infinitely large population should not be subject to evolutionary drift, or otherwise have reduced variance in the viable long term states. Figure \ref{figure_1} suggests that this is not so in the uniform mutation regime. Indeed, the entropy rate need not vanish in the large population limit, as we will see for the neutral landscape, but for the same landscape with the boundary regime, the entropy rate does vanish. Consider the neutral landscape with the boundary mutation regime for arbitrary $N$ and $\mu$. A straightforward calculation shows that
\begin{align}
 s_0 = s_N &= \left(2 + \mu N^2 \sum_{i=1}^{N-1}{\frac{1}{i (N-i)}}\right)^{-1}\\
 s_j &= \left(2 + \mu N^2 \sum_{i=1}^{N-1}{\frac{1}{i (N-i)}}\right)^{-1}\frac{\mu N^2}{j (N-j)}
 \label{neutral_boundary}
\end{align}
As expected from Theorem 2, it is still the case that as $\mu \to 0$, $s_0 = s_N \to 1/2$ and $s_j \to 0$ for the boundary mutation regime. The summation $\sum_{i=1}^{N-1}{\frac{1}{i (N-i)}} = 2 h_{N-1} / N \approx 2\log{N} / N$, where $h_n$ is the $n$th harmonic number. For large $N$ and fixed $\mu$, $p_i \to 0$ for all $i$. However, the $p_i$ do not converge to zero at the same rates asymptotically. For $i \neq 0, N$, $p_i \thicksim 1/\log{N}$ near the boundaries and as $p_i \thicksim 1/(N\log{N})$ when $i \approx N/2$, $i=0, N$. Nevertheless, it is the case that $H(P) \to 0$ as $N \to \infty$, which can be shown with a tedious but direct calculation using Sterling's approximation and the fact that $h_n \approx \log n$.

For the uniform mutation regime, the entropy rate does not approach zero as $N \to \infty$. In this case, the interior probabilities are larger, with $i \approx N/2$ being the maximum. There is a significant contribution to the entropy rate from this state (and nearby states) because the transition probabilities at this state are approximately $T_i = (1/4, 1/2, 1/4)$ for large $N$, so there is a contribution of $H(T_i) = 3/2 \log{2}$ to the entropy rate (weighted by the stationary distribution). For fixed $\mu$ and large $N$, the entropy rate converges to this value. As an illustration, consider the following example for the uniform mutation case with $\mu=1/2$. Then the stationary distribution is

\begin{equation}
 s_i = 2^{-N} \binom{N}{i},
\label{stationary_mu_one_half} 
\end{equation}
which is maximal at $i=N/2$, rather than at $i=0, N$ for the same process with only boundary mutations. Hence these processes can have very different stationary distributions for fixed $\mu$ and $N$ despite their similarity in definition. The entropy rate is
\[ H(P) = \sum_{i}{ \frac{1}{2^N} \binom{N}{i} H\left(\frac{N-i}{2N}, \frac{1}{2}, \frac{i}{2N} \right) },\]
which approaches $3/2 \log 2$ as $N \to \infty$.

For this example, as a function of $N$ the entropy rate is strictly increasing, since the stationary distribution is concentrating on the center where the entropy of the transition probabilities is largest. The limit holds for all fixed $\mu$ for the neutral landscape. Equations \ref{s_0} and \ref{s_j} imply that $s_j = s_{N-j}$ for all $j$ and that $s_0 < s_1 < \cdots < s_{\left \lfloor \frac{N}{2} \right \rfloor} = s_{\left \lceil \frac{N}{2} \right \rceil} > \cdots > s_{N-1} > s_N$. Then note that for $i < N / 2$, $p_{i+1} / p_i$ has higher order dependence on $N$ as $i$ approaches the central state(s). Similarly for $i > N / 2$, so the stationary distribution is increasing concentrated on the central state(s) as $N \to \infty$, giving an increasing entropy rate.

\subsection{Asymmetric mutation probabilities} So far we have only considered explicit examples where $\mu_{AB} = \mu = \mu_{BA}$. While all proper choices of $\mu_{AB}$ and $\mu_{BA}$ give a maximum entropy rate as $N$ varies, in the case of a neutral landscape, the maximum is largest when the mutation rates are equal. This is simply because for unequal mutation rates, the stationary distribution is no longer concentrated on the states with the largest transition entropies. More precisely, in Equations \ref{s_j} and \ref{s_0}, the factors corresponding to  $T_{1 \to 0} = \mu_{AB}$ and $1 / T_{N \to N-1} = 1 / \mu_{BA}$ no longer cancel, which has the effect of shifting the stationary distribution toward one of the boundary states depending on the value of $k = \mu_{BA} / \mu_{AB}$. It is possible to solve for the stationary distribution in the boundary regime for the neutral landscape as above (compare to Equations (\ref{neutral_boundary})):
% 
% \begin{align*}
%  s_0 &= \left(1 + \mu_{AB} N^2 \sum_{i=1}^{N-1}{\frac{1}{i (N-i)}} + \frac{\mu_{AB}}{\mu_{BA}}\right)^{-1}\\
% %  s_j &= \frac{1}{1 + \mu_{AB} N^2 \sum_{i=1}^{N-1}{\frac{1}{i (N-i)}} + \frac{\mu_{AB}}{\mu_{BA}}}\frac{\mu_{AB} N^2}{j (N-j)} \\
%  s_j &= s_0\frac{\mu_{AB} N^2}{j (N-j)} \\
%  s_N &= s_0\frac{\mu_{AB}}{\mu_{BA}}\\
% \end{align*}
\begin{align*}
 s_0 &= \left(1 + \mu_{AB} N^2 \sum_{i=1}^{N-1}{\frac{1}{i (N-i)}} + \frac{1}{k}\right)^{-1}\\
%  s_j &= \frac{1}{1 + \mu_{AB} N^2 \sum_{i=1}^{N-1}{\frac{1}{i (N-i)}} + \frac{\mu_{AB}}{\mu_{BA}}}\frac{\mu_{AB} N^2}{j (N-j)} \\
 s_j &= s_0\frac{\mu_{AB} N^2}{j (N-j)} \\
 s_N &= \frac{s_0}{k}\\
\end{align*}
\noindent Although the analytic calculation is messier for the uniform mutation regime, numerical computations indicate the unequal mutation probabilities give the same tendency to shift the stationary distribution away from the central states (where the transition entropies are larger). %We will revisit this topic briefly in the next section.

The two mutation regimes both have their merits: the boundary regime \cite{claussen2005non} is the simplest approach that guarantees the existence of a stationary distribution for most game matrices and is typically easier to generate analytic results for. The uniform mutation regime\cite{fudenberg2004stochastic} is perhaps a more realistic model of mutation, but has more complex transition probabilities. In any case, the resulting processes are of different character in certain parameter ranges, particularly in their large population behavior. Figure \ref{figure_kl_regimes} shows the KL-divergence between the stationary states of the uniform and boundary regimes for the neutral landscape for a variety of parameters.

\begin{figure}
    \centering
    \includegraphics[width=0.9\textwidth]{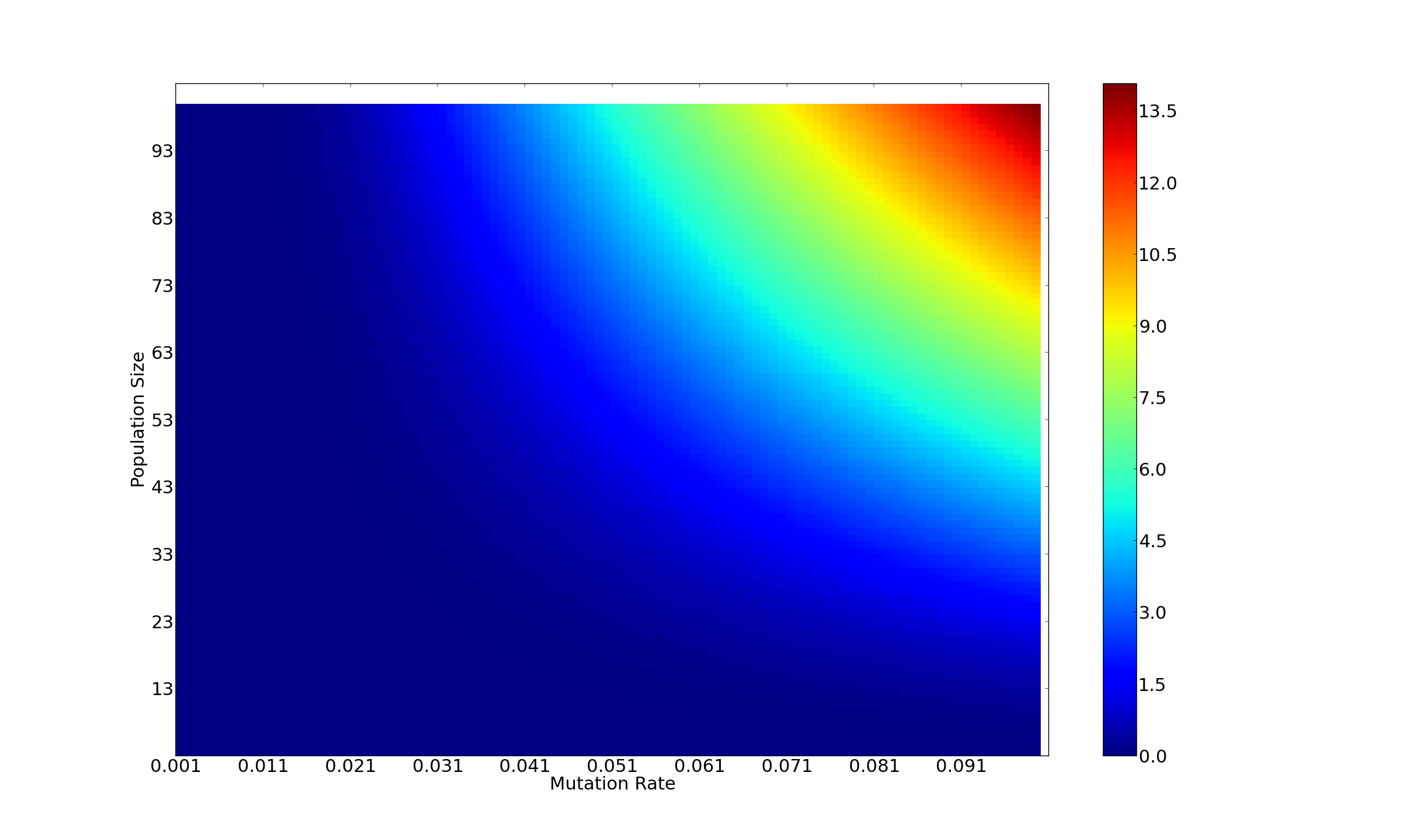}
    \caption{KL-divergence (equation \ref{kl_divergence}) for the stationary distributions of Moran processes for the neutral fitness landscape computed with the boundary regime and the uniform regime. For small values of $\mu$ (horizontal axis) relative to $N$ (vertical axis), i.e. $N \mu << 1$, the distance between the stationary distributions is small, in accordance with Theorems 1 and 2. For larger population sizes and fixed $\mu$, the stationary distributions diverge substantially. The hyperbolic curves of constancy are given by $C = N \mu$ for various constants $C$.}
    \label{figure_kl_regimes}
\end{figure}

\subsection{Maximum Entropy Rate for the Moran process} An obvious question resulting from this section is simply: what is the maximum entropy rate any Moran process can achieve? The maximum theoretical value is $\log 3$, and the entropy rate is limited by the maximum transition entropy. Suppose that for some $i$ that $T_{i \to i+1} = T_{i \to i-1}$. Then it is a simple algebraic exercise to show, for boundary mutations, that $f_A(i) = f_B(i)$ and the transition distribution is $T_i = \left(\frac{i (N-i)}{N^2}, \frac{(N-i)^2 + i^2}{N^2}, \frac{i (N-i)}{N^2}\right)$, which is maximal and equal to $3/2 \log 2$ when $i = N /2$. (Also it is true that if $f_A(i)=f_B(i)$, the same distribution holds.) So it is not possible for the transition distribution to be $(1/3, 1/3, 1/3)$, which would produce that maximum value of $\log 3$. Similarly equating other transitions leads to entropies less than $3/2 \log 2$. While it may seem possible to contrive the transitions to whatever distribution one desires, once one specifies the four variables needed for any one of the transitions, the other two are determined. In fact, $3/2 \log 2$ is the maximum possible value for any mutation regime, population size, and fitness landscape, so long as a unique stationary distribution exists (proof in Appendix). This means that the neutral landscape achieves the largest possible entropy rate, and so verifies the intuition that it should have the largest inherent uncertainty in long run population behavior.

\section{Neutral Evolution: Wright-Fisher Process}

Stationary distributions for the Wright-Fisher process are difficult to compute analytically in general. We can develop bounds for some important special cases. First, notice that for any state $i \neq 0, N$, the transition distribution $T_i = (T_{i \to 0}, \ldots, T_{i \to N})$ is a binomial distribution. Although a convenient closed form for the entropy of a binomial distribution does not exist, there is a useful approximation given by the normal approximation to the binomial (the de-Moivre-Laplace central limit theorem). Given a binomial distribution with probability $p$ and $N$ trials, the entropy is approximately
\[H\left(\text{binomial}(N, p)\right) = \frac{1}{2} \log \left(2 \pi e N p (1-p)\right) + O\left(\frac{1}{N}\right).\]
For fixed $N$, the maximum value occurs when $p=1/2$. For states $i=0$ and $i=N$, the distributions are just $(1-\mu, \mu, 0, \ldots, 0)$ and $(0, \ldots, 0, \mu, 1-\mu)$.

\begin{figure}
    \centering
    \includegraphics[width=0.45\textwidth]{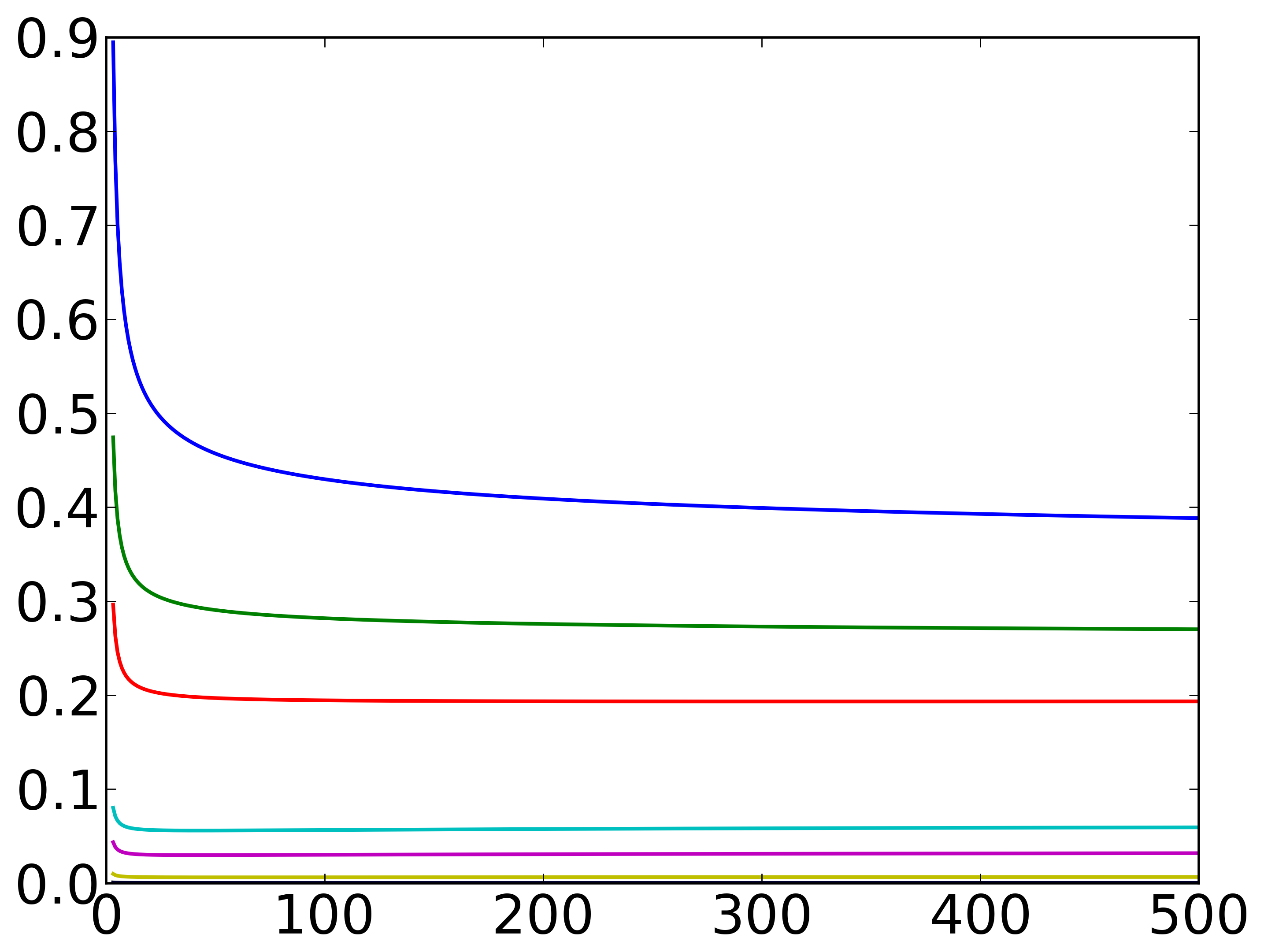}
    \includegraphics[width=0.45\textwidth]{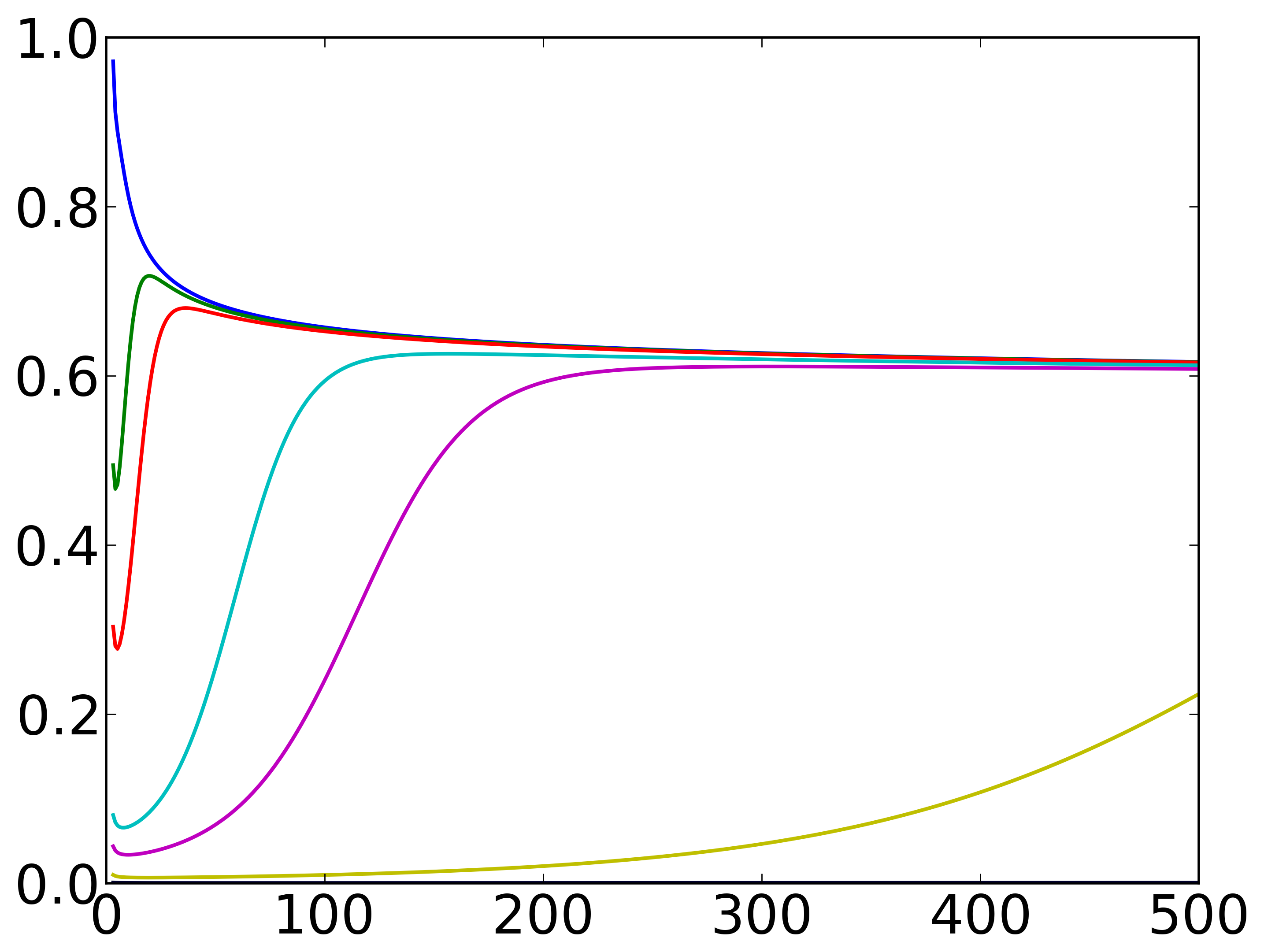}
    \caption{Scaled Entropy Rate vs. Population Size $N$ for $\mu_{AB} = \mu_{BA} \in \{0.5, 0.1, 0.05, 0.01, 0.005, 0.001, 0.0001\}$ (top to bottom) with a neutral fitness landscape for the Wright-Fisher process. Entropy rates are divided by $\log N$. Left: Mutations only at the boundary states (boundary regime); the entropy rate approaches a value less than 1/2 as $N \to \infty$. Right: Mutations for all states (uniform regime); the entropy rates approach $1/2$ as $N \to \infty$. Compare to Figure \ref{figure_1}.}
    \label{figure_wright_fisher_neutral}
\end{figure}

To compare to the Moran process, let us consider some similar scenarios to the previous section. Consider a very special case where the process is governed by uniform mutations with $\mu = \mu_{AB} = \mu_{BA} = 1/2$ (for any fitness landscape). Then the transition probabilities are given by $T_{i \to j} = \binom{N}{j}2^{-N}$ and so we have a binomial distribution with $p=1/2$ for all $i \neq 0, N$. These distributions are largest at the central state(s) and decrease monotonically away from the center. For $N \mu >> 1$ the stationary distribution is concentrated on the central states, and the entropy rate is
\begin{equation}
 H(P) \approx \frac{1}{2} \log \left(\frac{\pi e}{2}\right) + \frac{1}{2} \log{N}
 \label{wright_fisher_bound}
\end{equation}
\noindent which is less than the theoretical maximum of $\log N$, but still unbounded in $N$. Computational results verify that this formula is an excellent approximation of the entropy rate even for small $N \approx 15$. This is the maximum attainable entropy rate for the Wright-Fisher process since it is the maximum possible entropy for a binomial distribution. For the boundary mutation regime, the entropy rate appears to be increasing for fixed $\mu$ as $N$ increases, but bounded by a smaller value. See Figure \ref{figure_wright_fisher_neutral} for plots of entropy rates for various $\mu$ and $N$ for the neutral landscape.

For the neutral fitness landscape but with boundary mutations only, the transition probability distributions $T_i$ are given by a binomial distribution on $p=i/N$. While at the central point $i=N/2$ the transition entropy is the same as for the uniform mutation case, the entropies are decreasing away from the central point. Moreover, the stationary distribution is concentrated at the boundary states $i=0$ and $i=N$, and so the entropy rate is much smaller. For fixed $N$, computations indicate the stationary distribution approaches $1/2$ at the boundary states and the entropy rate tends to the binary entropy $H((\mu, 1-\mu))$ plus contributions from the boundary adjacent states, which approaches zero as $\mu \to 0$. This also appears to be the case for fixed $N$ and $\mu \to 0$ in the uniform mutation case as well. This is because as for $N \mu << 1$, the transition distributions $T_i$ approach binomials with $p \approx i / N$, which skews the activity of the process toward the boundary states in the case of uniform mutations. For boundary mutations, small $\mu$ means that the boundary states are more absorptive, and so the stationary distribution has weight concentrated at the boundaries. Hence just as for the Moran process, the behavior is very similar for small $\mu$ for the different mutation regimes, and the entropy rate goes to zero as $\mu$ does. The analogous plot from Figure \ref{figure_kl_regimes} for the Wright-Fisher process is nearly identical in shape (but not magnitude), so we will omit it.

We end this section with the following conjectures. A suitable theorem for the stationary distribution analogous to the theorem from \cite{fudenberg2004stochastic} used for Theorem 1 in the previous section would partially establish the second conjecture. Note that the upper bound above (Equation \ref{wright_fisher_bound}) is not among the conjectures as the discussion contains sufficient proof.
\subsection{Conjectures}
Based on computation evidence and comparison with the Moran process, we state the following conjectures for the entropy rate of the Wright-Fisher Process.

\noindent \begin{conjecture}
    \begin{enumerate}
    \item For the boundary mutation regime with $\mu_{AB} = \mu = \mu_{BA}$ and neutral landscape, there exists a constant $C_{\mu} \leq 1/2$ such that $\lim_{N \to \infty}{H(P) / \log{N}} = C_{\mu}$.\\
    \item Let $\mu_{AB}=\mu$ and $\mu_{BA} = k \mu$. For both the uniform and boundary mutation regimes, $\lim_{\mu \to 0}{H(P)} = 0$.
    \end{enumerate}
\end{conjecture}

\section{Constant Fitness: Moran Process}

Now we turn back to the Moran process and investigate non-neutral fitness landscapes. Consider a population in which one type has constant relative fitness $r$, i.e. for the game matrix $a = b = r$ and $c = d = 1$, corresponding to the classical Moran process.
Figure \ref{figure_moran_constant_fitness} shows the relationship between $N$ and $r$ for various $\mu > 0$. In particular, it is clear from the heatmaps that the maximum values for a given $N$ occur for $r = 1$ (i.e. tracing horizontally along any of the heatmaps). Moreover, the entropy rates seem to behave the same as the neutral landscape for large $N$ for both mutation regimes. The behavior for small $\mu$ is governed by Theorems 1 and 2.

\begin{landscape}
\centering
\begin{figure}[h]
        \begin{subfigure}[b]{0.4\textwidth}
            \centering
            \includegraphics[width=\textwidth]{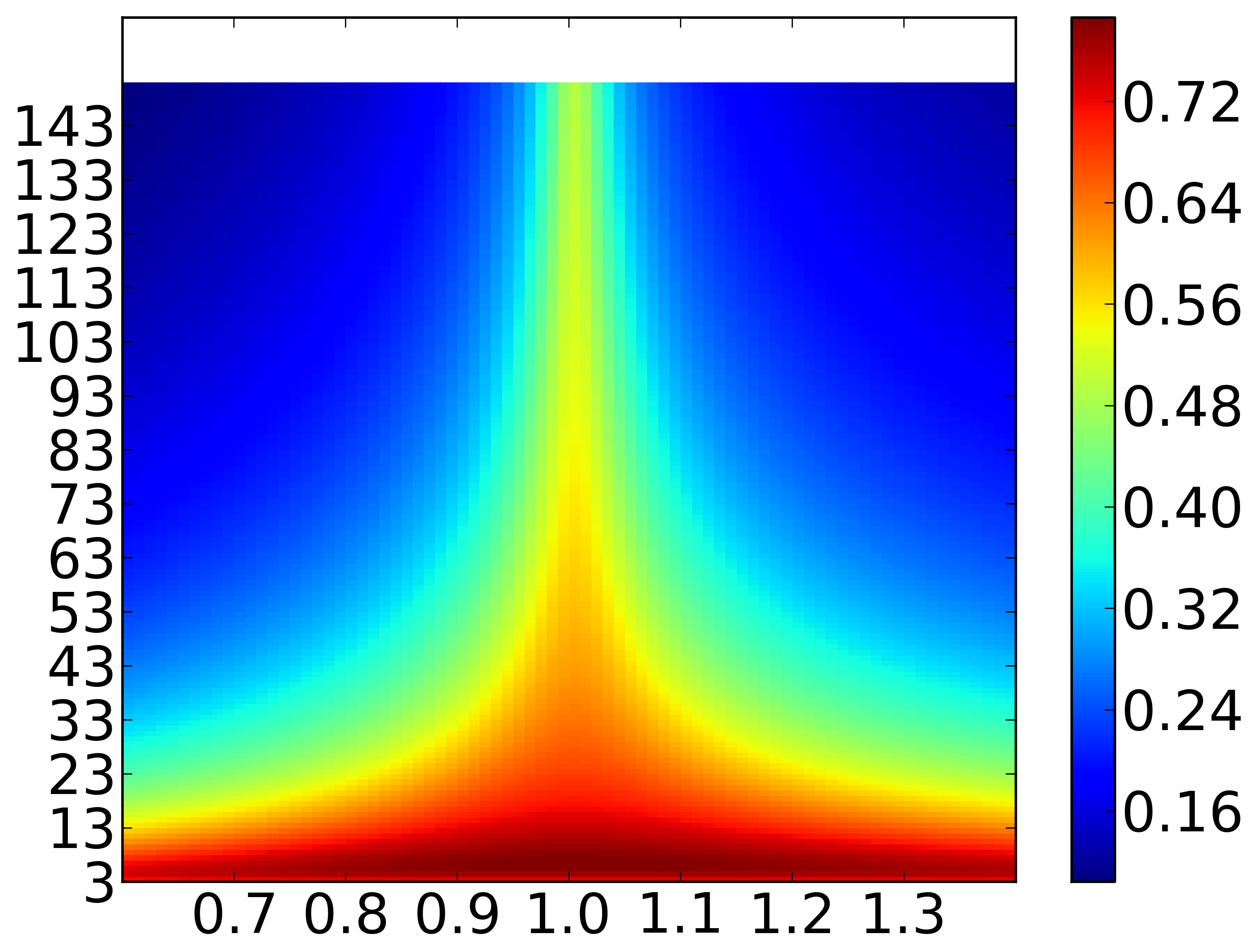}
%             \caption{}
%             \label{}
        \end{subfigure}%
        ~ %add desired spacing between images, e. g. ~, \quad, \qquad etc. 
        \begin{subfigure}[b]{0.4\textwidth}
            \centering
            \includegraphics[width=\textwidth]{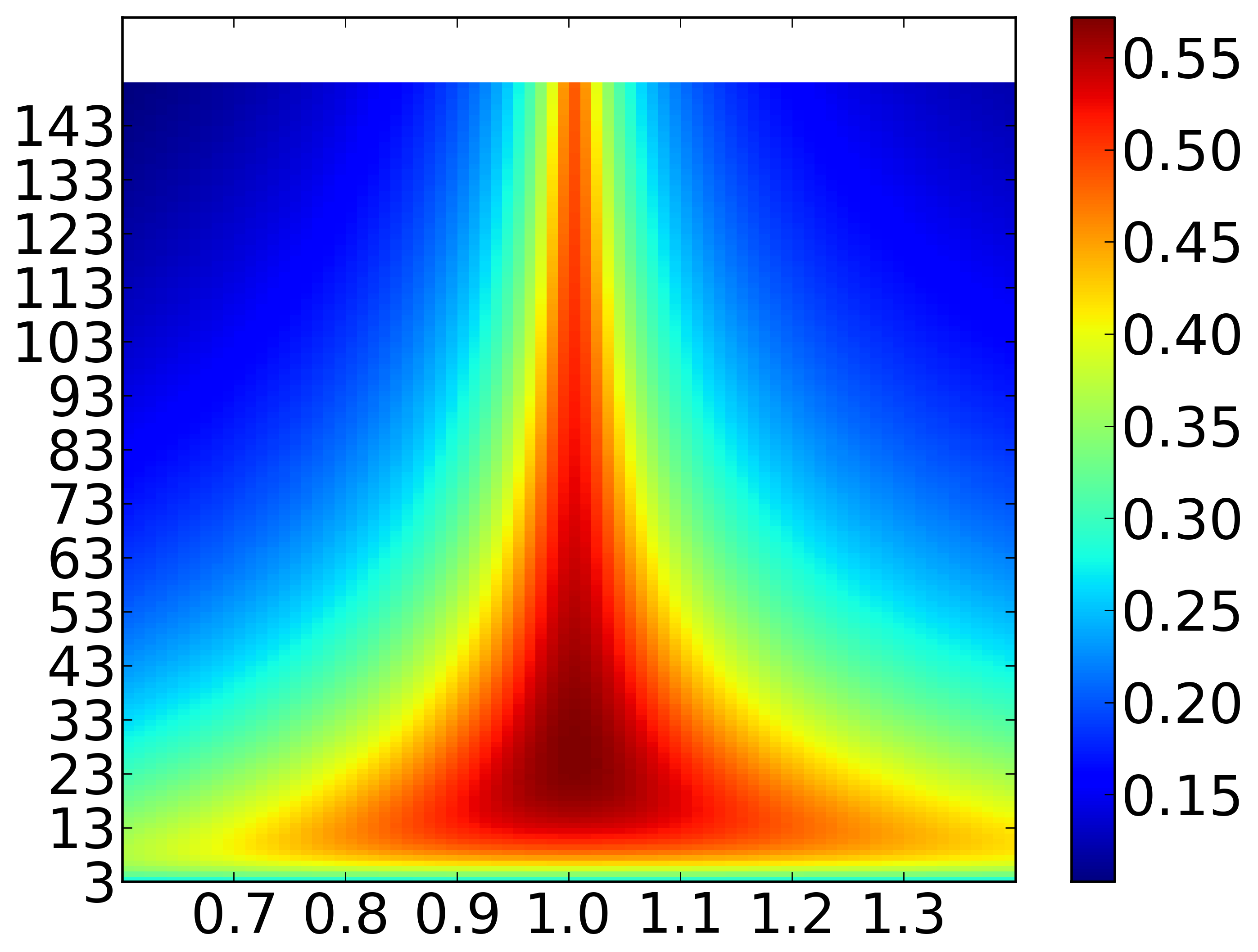}
%             \caption{}
%             \label{}
        \end{subfigure}
        \begin{subfigure}[b]{0.4\textwidth}
            \centering
            \includegraphics[width=\textwidth]{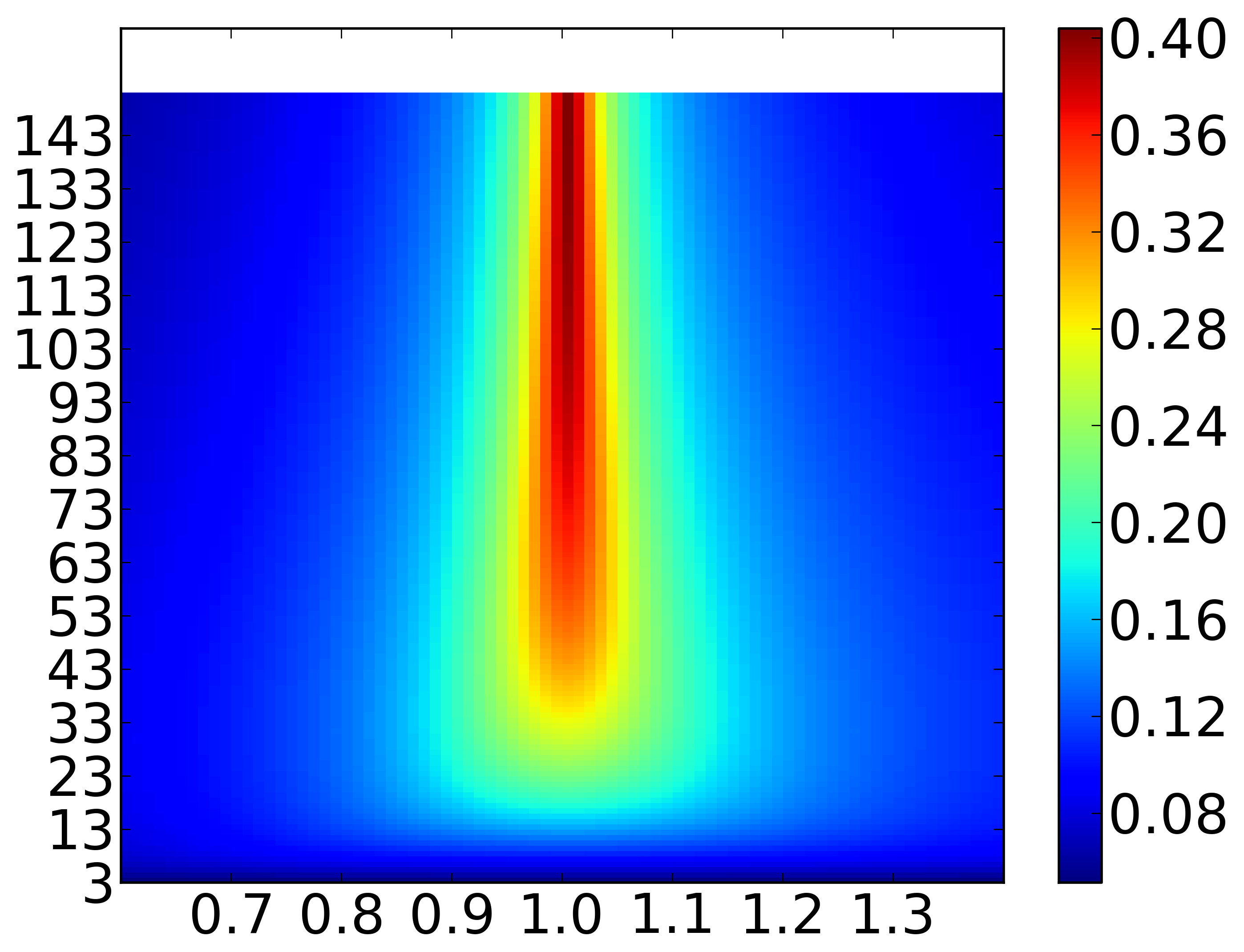}
%             \caption{}
%             \label{}
        \end{subfigure}
        \\
        \begin{subfigure}[b]{0.4\textwidth}
            \centering
            \includegraphics[width=\textwidth]{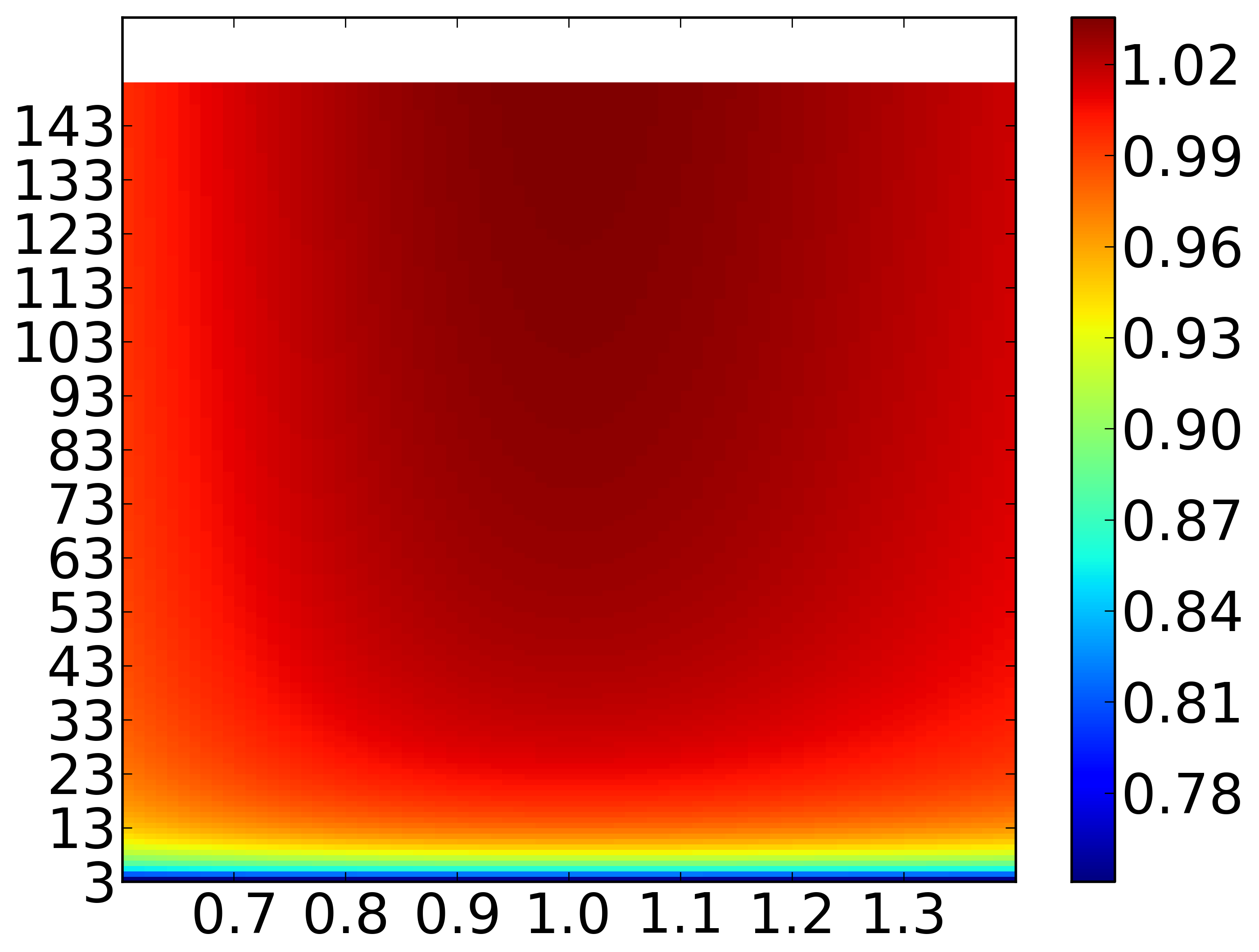}
            \caption{$\mu=0.2$}
%             \label{}
        \end{subfigure}%
        ~ %add desired spacing between images, e. g. ~, \quad, \qquad etc. 
        \begin{subfigure}[b]{0.4\textwidth}
            \centering
            \includegraphics[width=\textwidth]{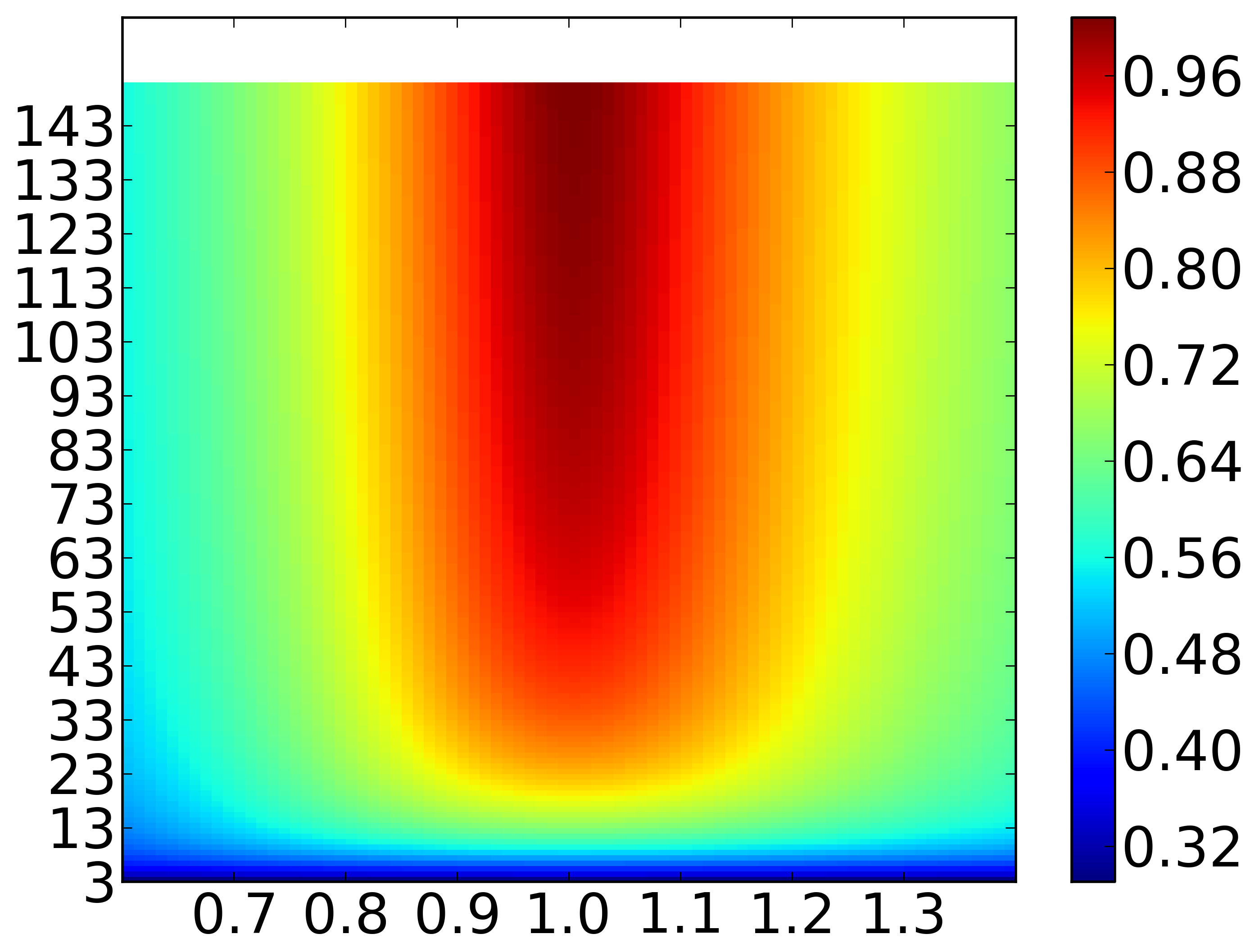}
            \caption{$\mu=0.04$}
%             \label{}
        \end{subfigure}
        \begin{subfigure}[b]{0.4\textwidth}
            \centering
            \includegraphics[width=\textwidth]{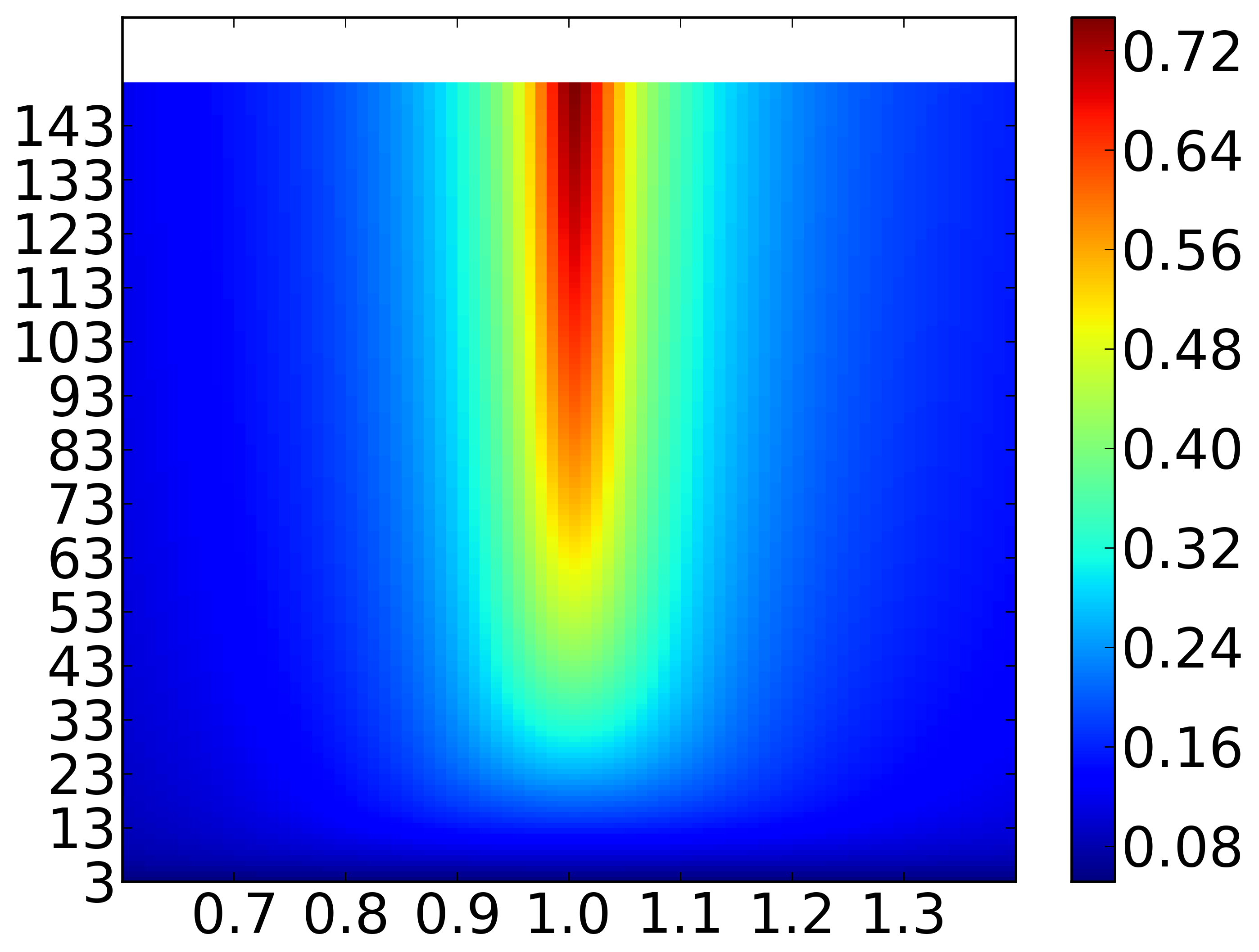}
            \caption{$\mu=0.005$}
%             \label{}
        \end{subfigure}        
        \caption{Entropy rate heatmaps for the Moran process, $r \in [0.6, 1.4]$ in the horizontal axis, $N \in [2, 150]$ on the vertical axis. Top row: Boundary mutation regime. Bottom Row: Uniform mutation regime. Though the plots look symmetric horizontally about $r=1$, they are not precisely so. Colorbars are not consistent across plots (for additional resolution). Entropy rates are generally smaller as $\mu$ decreases (left to right). For the top row, the entropy rate is eventually decreasing as $N$ increases; for the bottom row, the entropy rate increases as $N$ increases.}
        \label{figure_moran_constant_fitness}
\end{figure} 
\end{landscape}

For the boundary mutation case, we have that there is evidently least one local maximum for the entropy rate as the population varies (see Figure \ref{figure_1}). This appears to be true more generally. Figure \ref{moran_boundary_maximal_N} plots the population size $N$ that maximizes the entropy rate versus $r$ for various mutation rates. Although drift typically dominates evolution for small populations leading such populations to be thought of as ``more random'', the inherent randomness measured by the entropy rate can be higher for larger populations. This is because higher entropy rate comes from a balance of mutation, selection, and drift that cannot simply be reduced to population size. Indeed, small populations can be subject to so much stochastic variation from drift that they become more predictable in their long run behavior and hence have smaller entropy rates.

\begin{figure}
    \centering
    \includegraphics[width=0.8\textwidth]{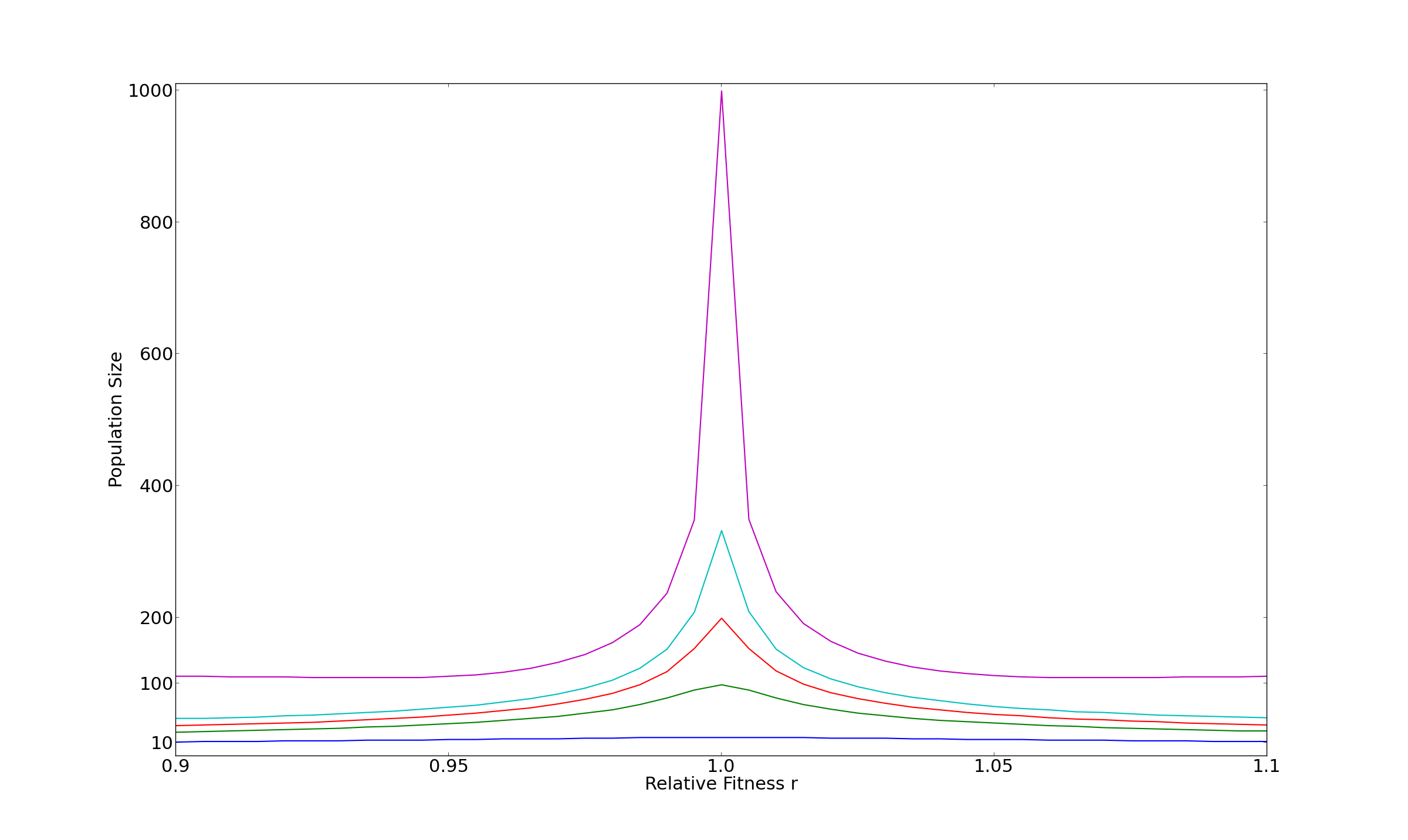}
    \caption{Mutation-drift balance: (Vertical) Population size $N$ for which the entropy rate is maximal versus the relative fitness $r$ for various fixed values of $\mu$ (boundary mutations). From Bottom to Top: $\mu = 0.05, 0.01, 0.005, 0.003, 0.001$. The maximum for $r=1$ appears to occur when $N \mu \approx 1$. The maximum entropy rate may occur for a large population size even though the entropy rate tends to zero as the population size gets large. These curves correspond to the maximum values in Figure \ref{figure_1}.}
    \label{moran_boundary_maximal_N}
\end{figure}

\subsection{Constant Fitness: Wright-Fisher Process}

Computational results for the Wright-Fisher process indicate similar behavior for constant fitness landscapes. Figure \ref{figure_wf_constant_fitness} has heatmaps for similar parameters as Figure \ref{figure_moran_constant_fitness} for the Wright-Fisher process. Interestingly, the Wright-Fisher process appears to be more tightly clustered to $r=1$ for the boundary mutation, and varies more sharply for uniform mutation.

\begin{landscape}
\centering
\begin{figure}[h]
        \begin{subfigure}[b]{0.4\textwidth}
            \centering
            \includegraphics[width=\textwidth]{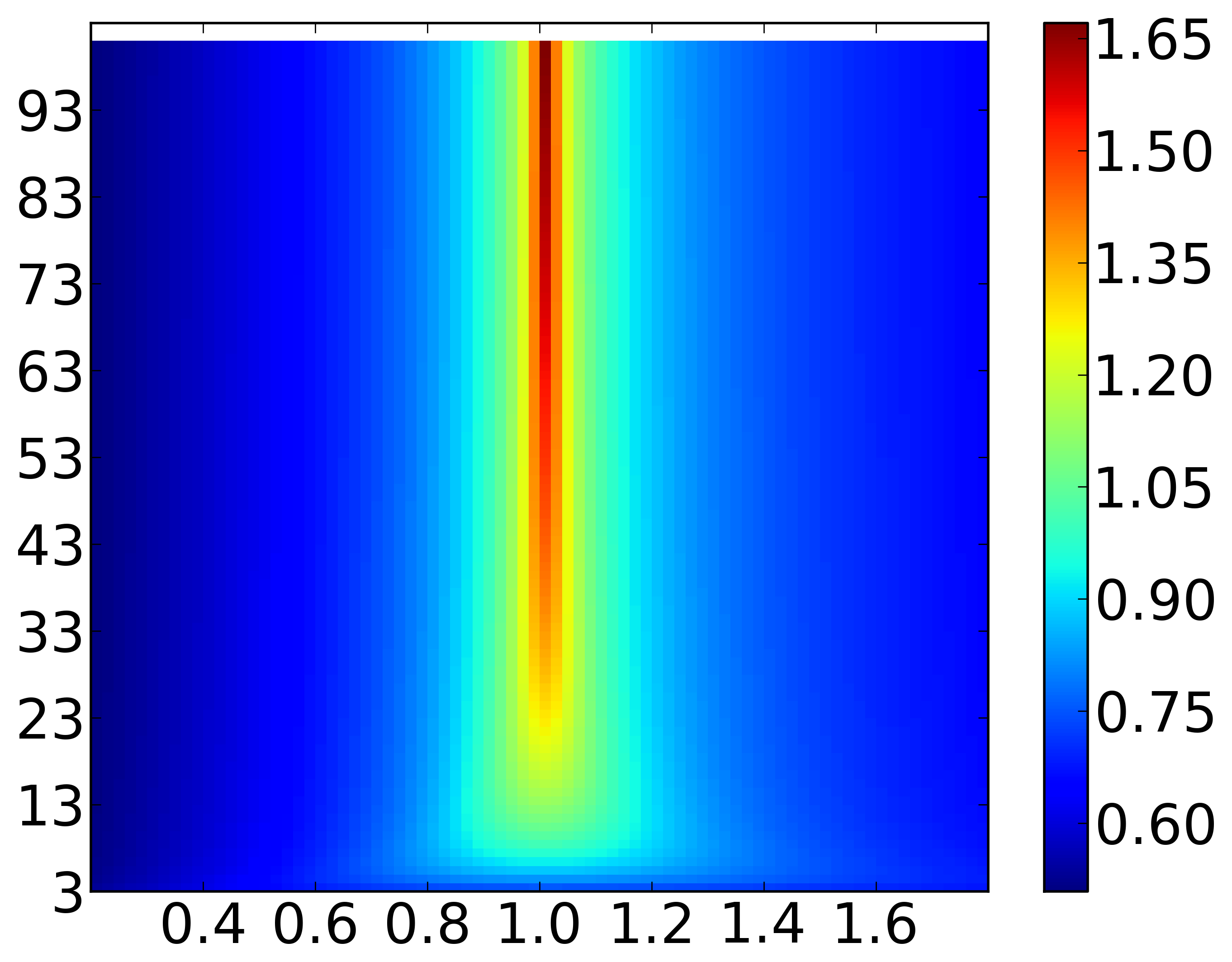}
%             \caption{}
%             \label{}
        \end{subfigure}%
        ~ %add desired spacing between images, e. g. ~, \quad, \qquad etc. 
        \begin{subfigure}[b]{0.4\textwidth}
            \centering
            \includegraphics[width=\textwidth]{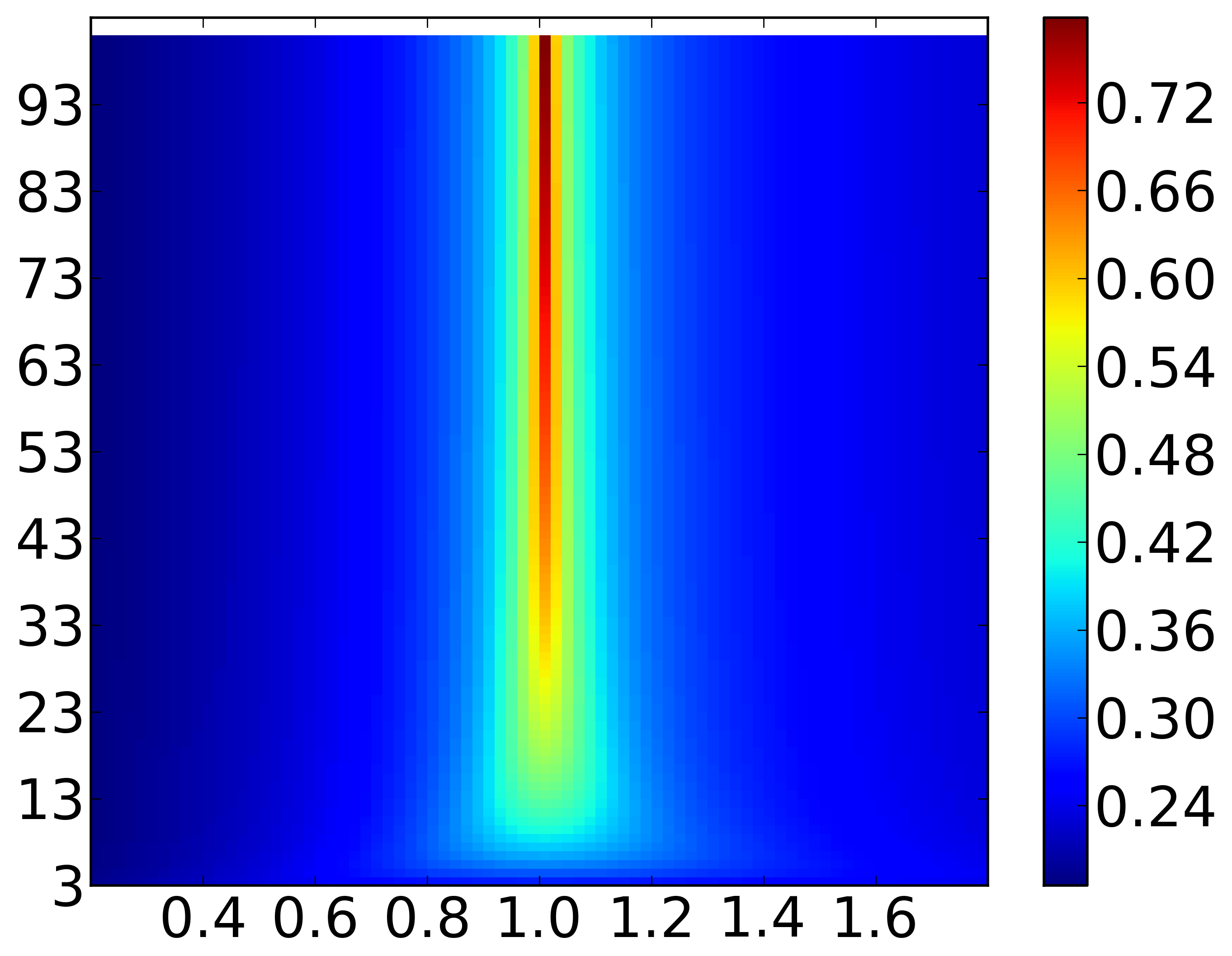}
%             \caption{}
%             \label{}
        \end{subfigure}
        \begin{subfigure}[b]{0.4\textwidth}
            \centering
            \includegraphics[width=\textwidth]{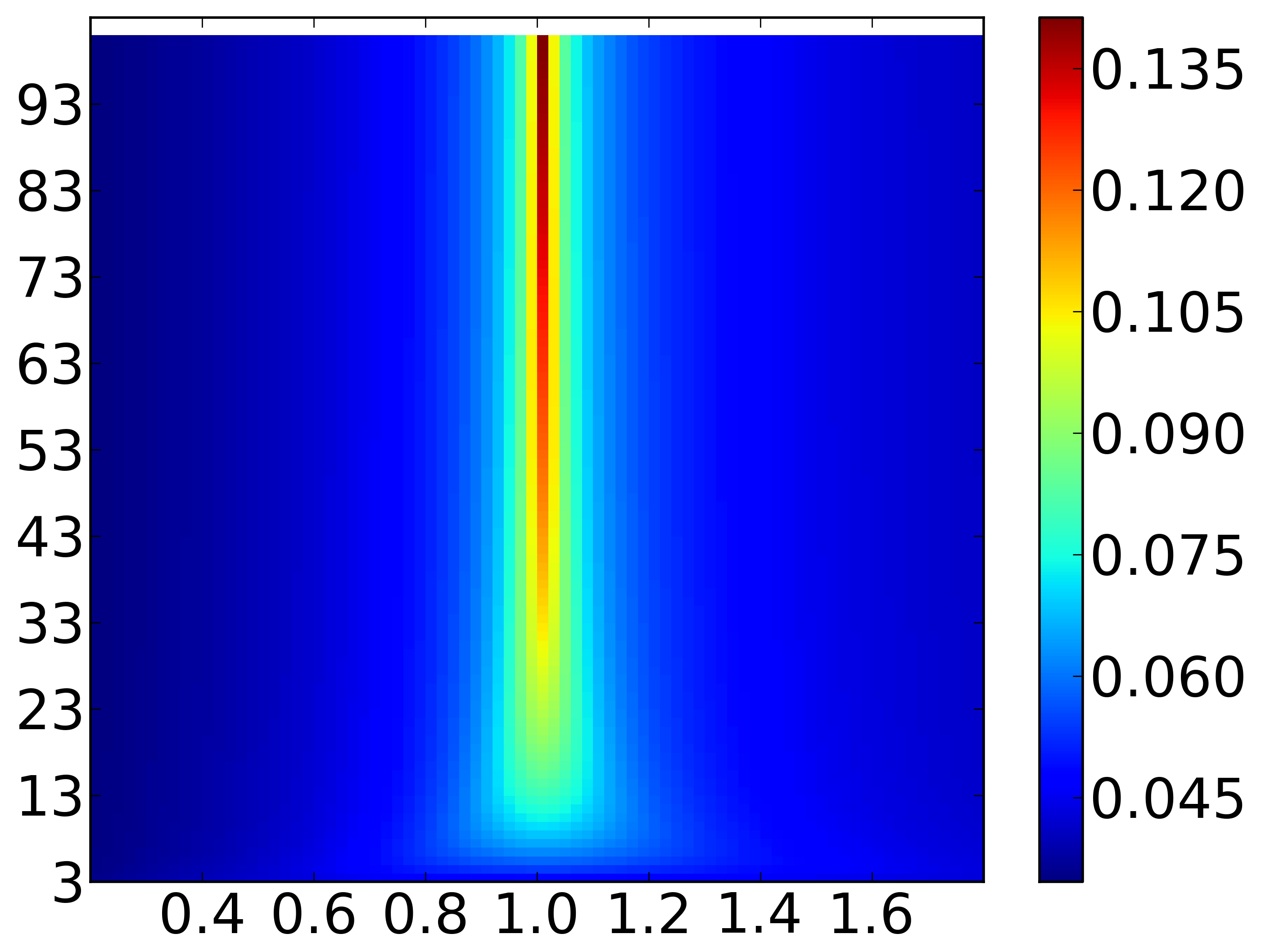}
%             \caption{}
%             \label{}
        \end{subfigure}
        \\
        \begin{subfigure}[b]{0.4\textwidth}
            \centering
            \includegraphics[width=\textwidth]{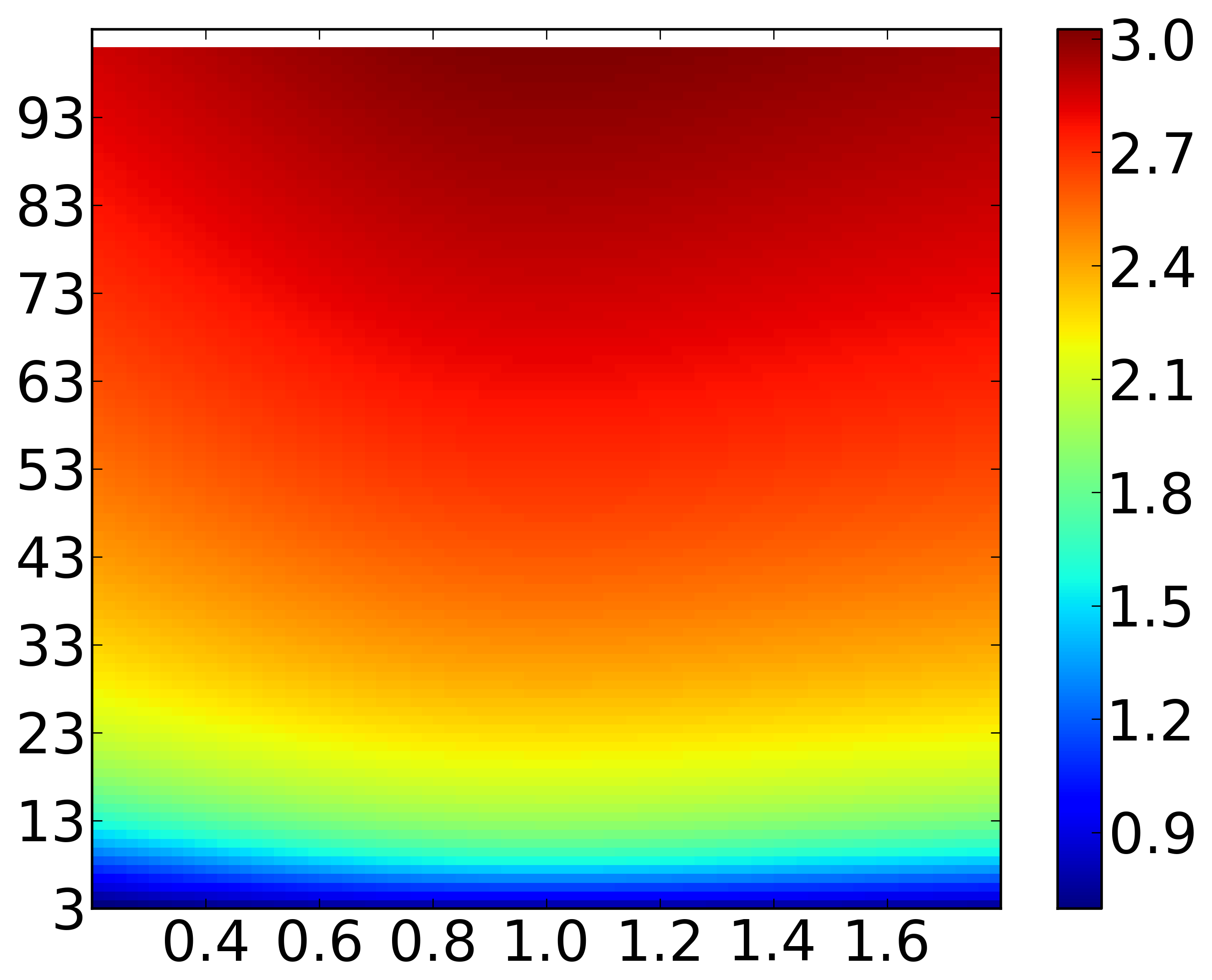}
            \caption{$\mu=0.2$}
%             \label{}
        \end{subfigure}%
        ~ %add desired spacing between images, e. g. ~, \quad, \qquad etc. 
        \begin{subfigure}[b]{0.4\textwidth}
            \centering
            \includegraphics[width=\textwidth]{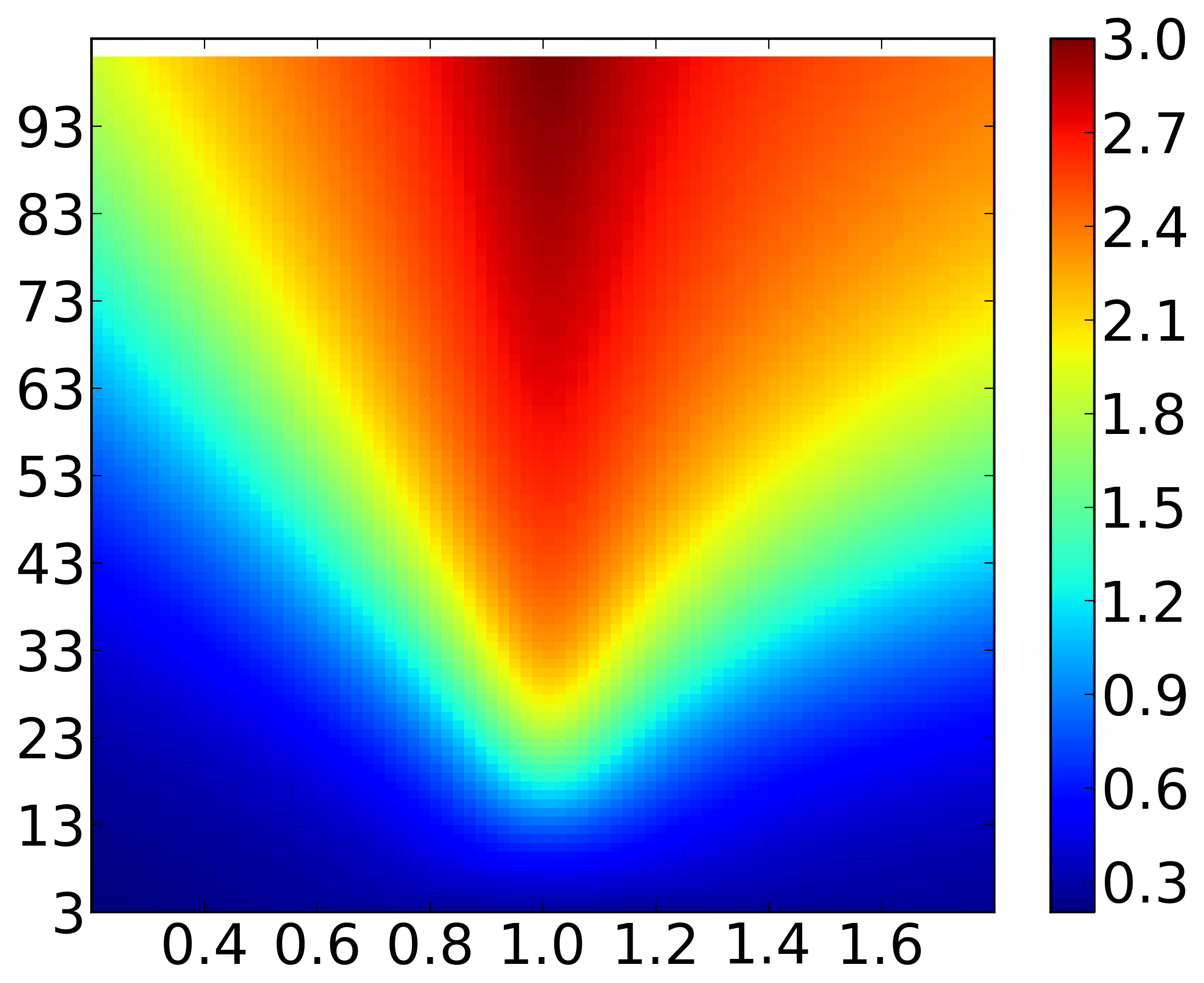}
            \caption{$\mu=0.04$}
%             \label{}
        \end{subfigure}
        \begin{subfigure}[b]{0.4\textwidth}
            \centering
            \includegraphics[width=\textwidth]{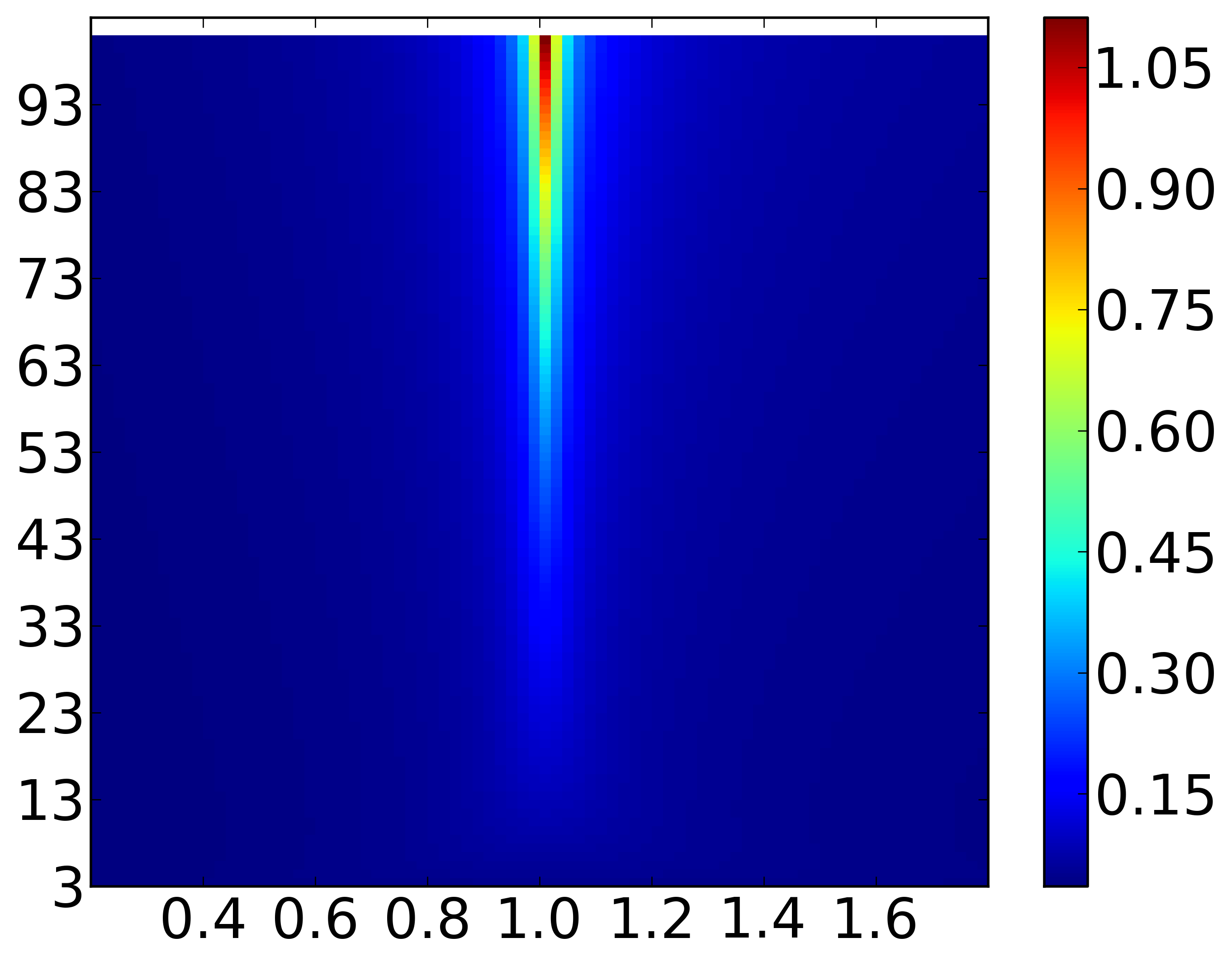}
            \caption{$\mu=0.005$}
%             \label{}
        \end{subfigure}        
        \caption{Entropy rate heatmaps for the Wright-Fisher process, $r \in [0.6, 1.4]$ in the horizontal axis, $N \in [2, 100]$ on the vertical axis. Top row: Boundary mutation regime. Bottom Row: Uniform mutation regime. Though the plots look symmetric horizontally about $r=1$, they are not precisely so. Colorbars are not consistent across plots (for additional resolution). Entropy rates are generally smaller as $\mu$ decreases (left to right). For the top row, the entropy rate is eventually decreasing as $N$ increases; for the bottom row, the entropy rate increases as $N$ increases.}
        \label{figure_wf_constant_fitness}
\end{figure} 
\end{landscape}

\section{Entropy Rates of $n$-fold Moran Process}

Transition probabilities of the $n$-fold Moran processes are typically not binomial distributions, even in the generational $k=N$ case. This is because the $N$ individuals are not necessarily chosen from the same population state like the $N$ individuals in the Wright-Fisher process. Depending on the values of the parameters, the entropy rates of the $n$-fold Moran process can be larger or smaller than those of the Wright-Fisher process for large $k$ with all other parameters the same. The same is true for $n$-fold processes for different values of $n$ but other parameters equal, even $n=1$ versus $n=2$. Numerical computations indicate this is heavily dependent on the value of the mutation rate.

\begin{figure}
    \centering
    \includegraphics[width=0.45\textwidth]{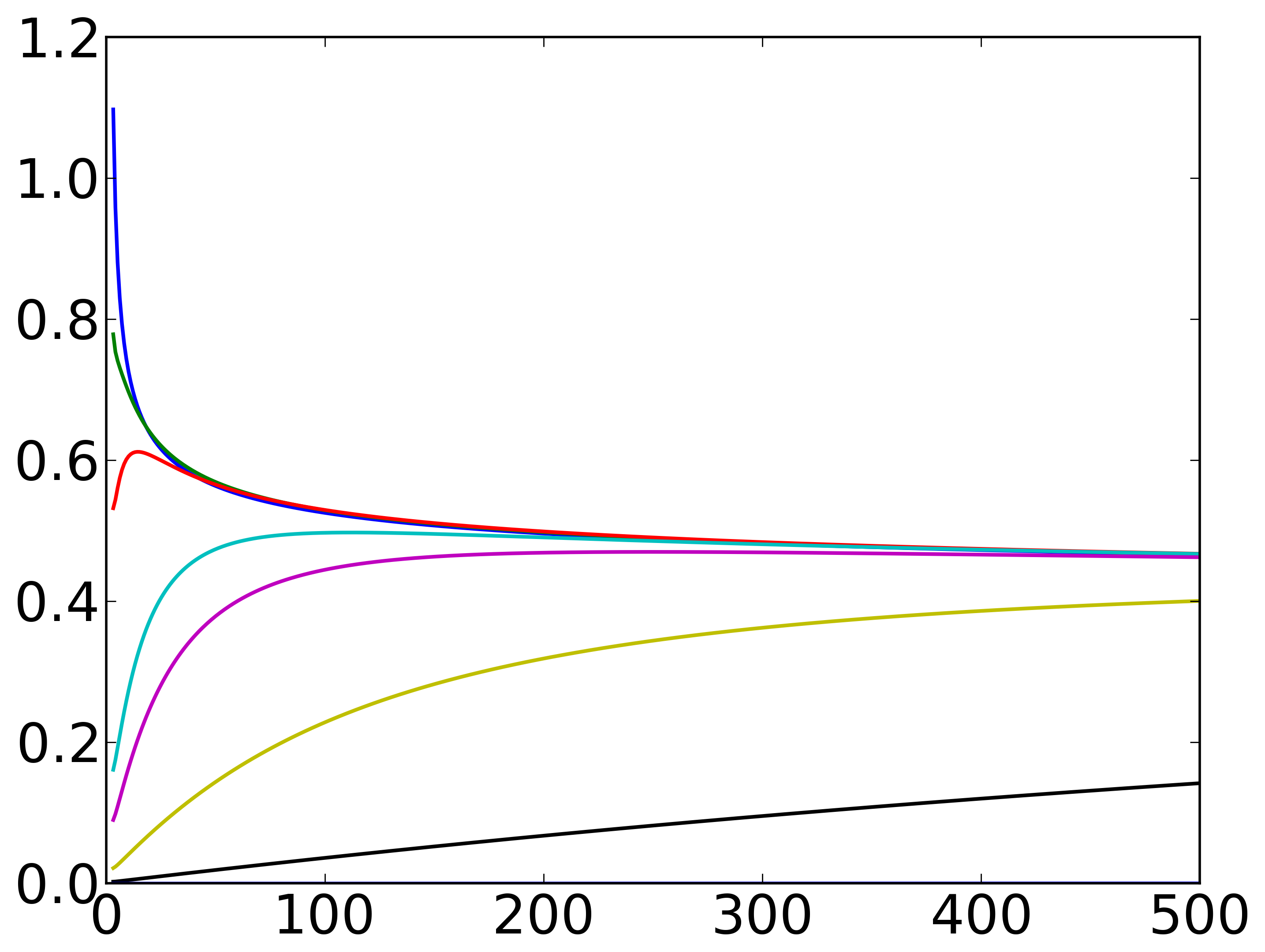}
    \includegraphics[width=0.45\textwidth]{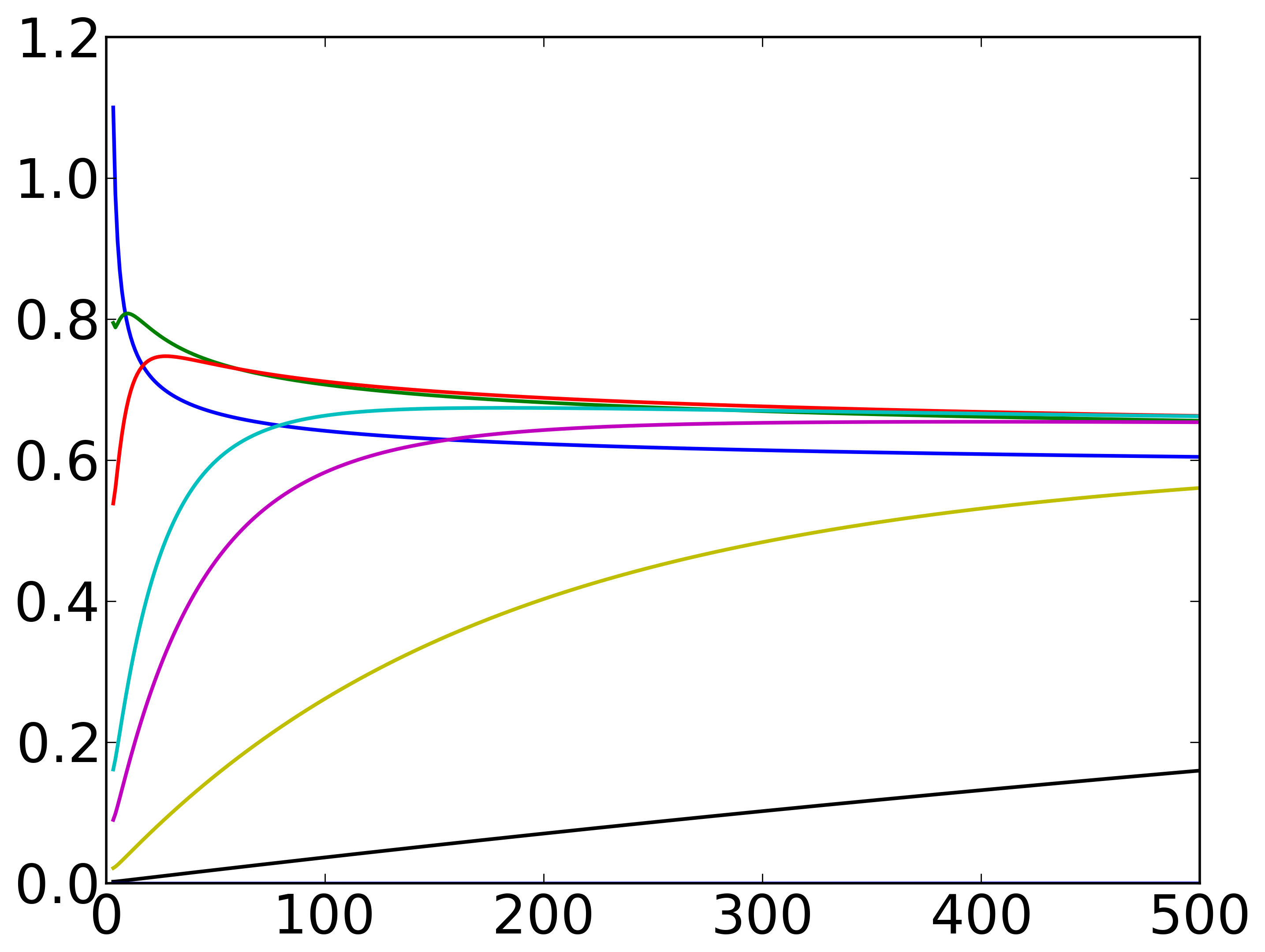}
    \caption{Scaled Entropy Rate vs. Population Size $N$ for $\mu_{AB} = \mu_{BA} \in \{0.5, 0.1, 0.05, 0.01, 0.005, 0.001, 0.0001\}$ (top to bottom) with a neutral fitness landscape for the $n$-fold Moran process. Entropy rates are divided by $\log N$. Left: Mutations only at the boundary states (boundary regime). Right: Mutations for all states (uniform regime). Compare to Figures \ref{figure_1} and \ref{figure_wright_fisher_neutral}.}
    \label{figure_n_fold_neutral}
\end{figure}

Once again the neutral fitness landscape appears to give the maximum value of the entropy rate. This simply because values of $r \neq 1$ will lead to the stationary distribution favoring one fixation state over the other, and lead to less spread out distributions. See Figure \ref{figure_n_fold_constant_fitness} and compare to Figures \ref{figure_moran_constant_fitness} and \ref{figure_wf_constant_fitness}. The entropy rates for the $N$-fold Moran processes are qualitatively similar to both the Moran process and the Wright-Fisher process. For $n \approx N / 2$, the entropy rates are very similar to those in Figure \ref{figure_n_fold_constant_fitness}.

\begin{landscape}
\centering
\begin{figure}[h]
        \begin{subfigure}[b]{0.4\textwidth}
            \centering
            \includegraphics[width=\textwidth]{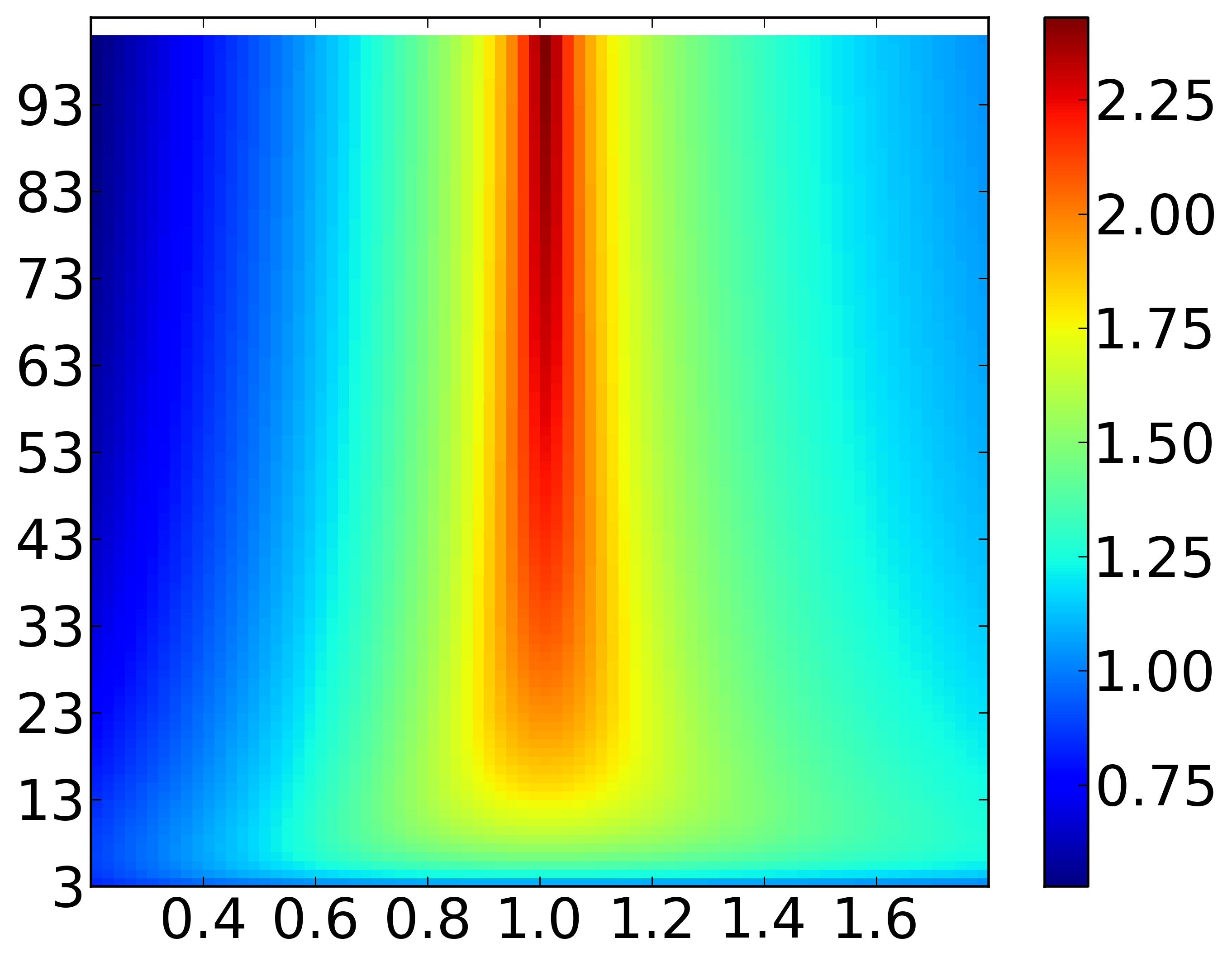}
%             \caption{}
%             \label{}
        \end{subfigure}%
        ~ %add desired spacing between images, e. g. ~, \quad, \qquad etc. 
        \begin{subfigure}[b]{0.4\textwidth}
            \centering
            \includegraphics[width=\textwidth]{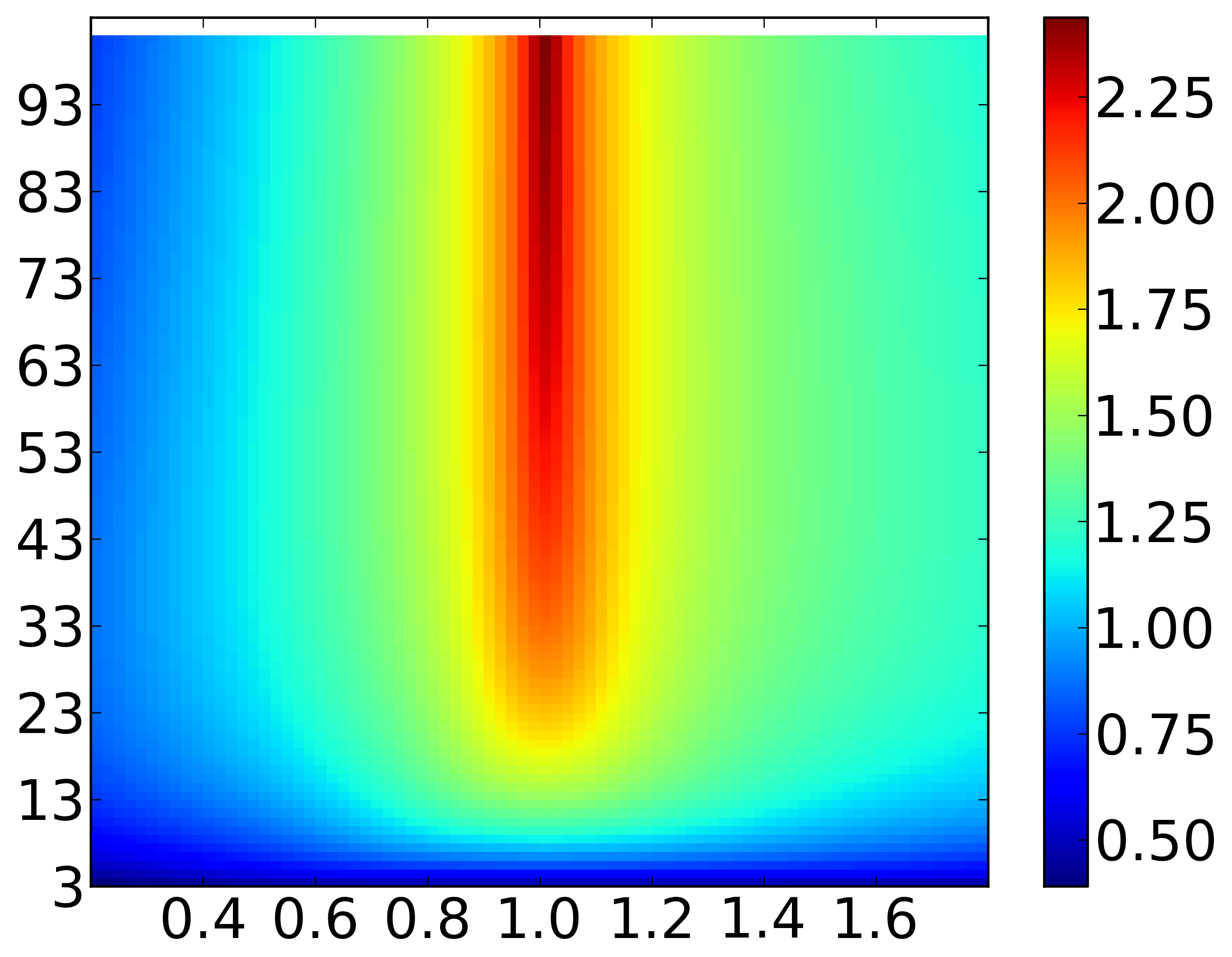}
%             \caption{}
%             \label{}
        \end{subfigure}
        \begin{subfigure}[b]{0.4\textwidth}
            \centering
            \includegraphics[width=\textwidth]{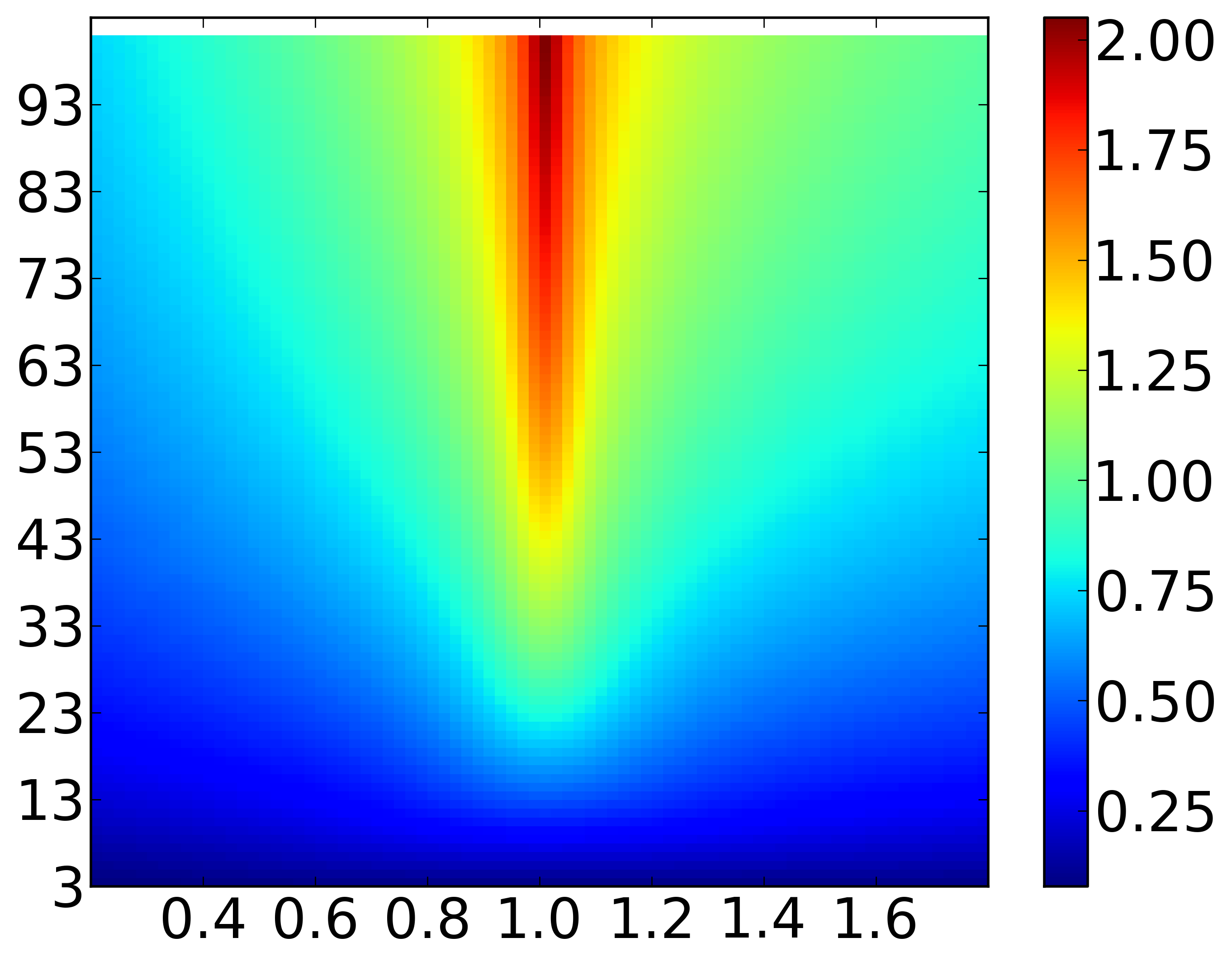}
%             \caption{}
%             \label{}
        \end{subfigure}
        \\
        \begin{subfigure}[b]{0.4\textwidth}
            \centering
            \includegraphics[width=\textwidth]{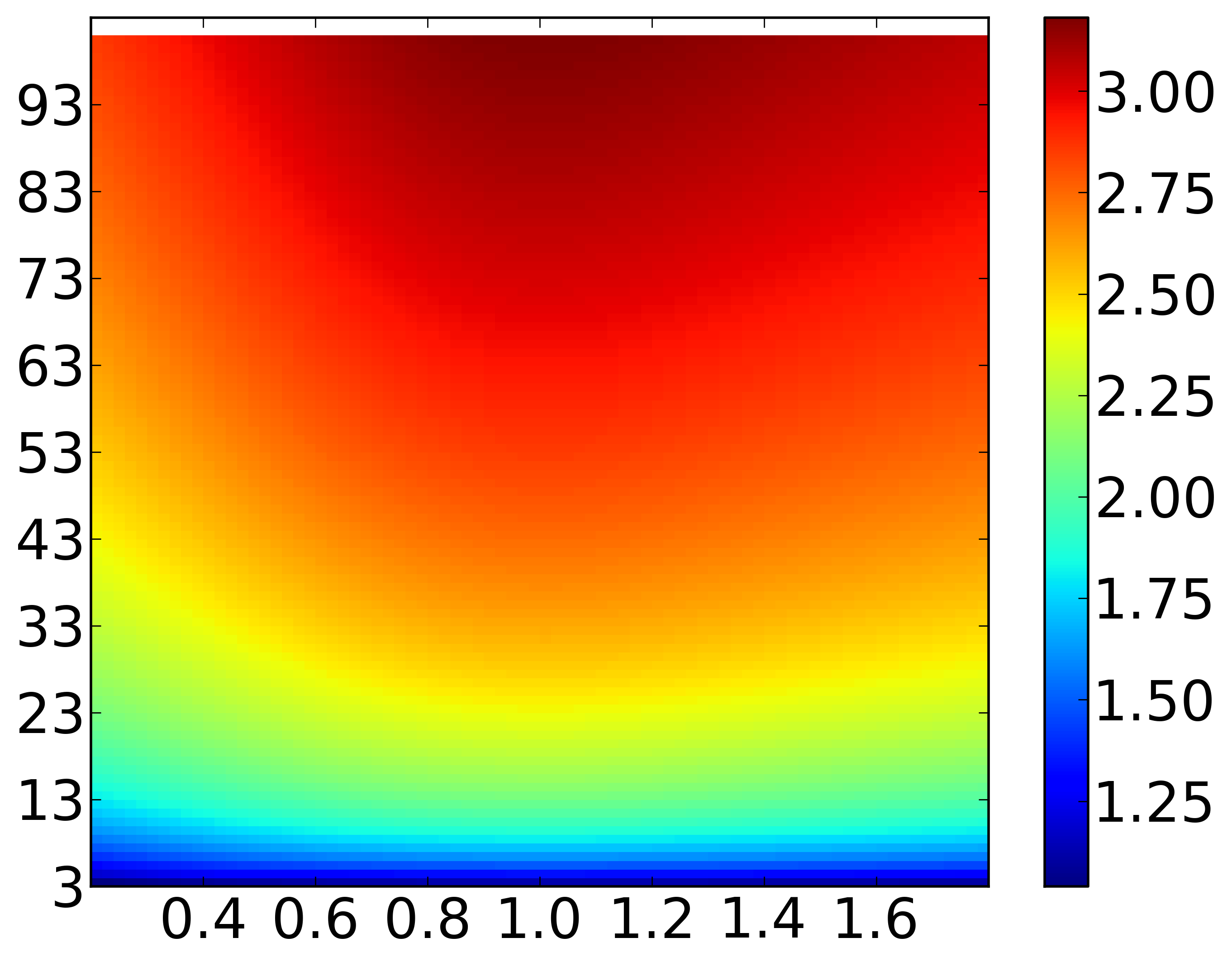}
            \caption{$\mu=0.2$}
%             \label{}
        \end{subfigure}%
        ~ %add desired spacing between images, e. g. ~, \quad, \qquad etc. 
        \begin{subfigure}[b]{0.4\textwidth}
            \centering
            \includegraphics[width=\textwidth]{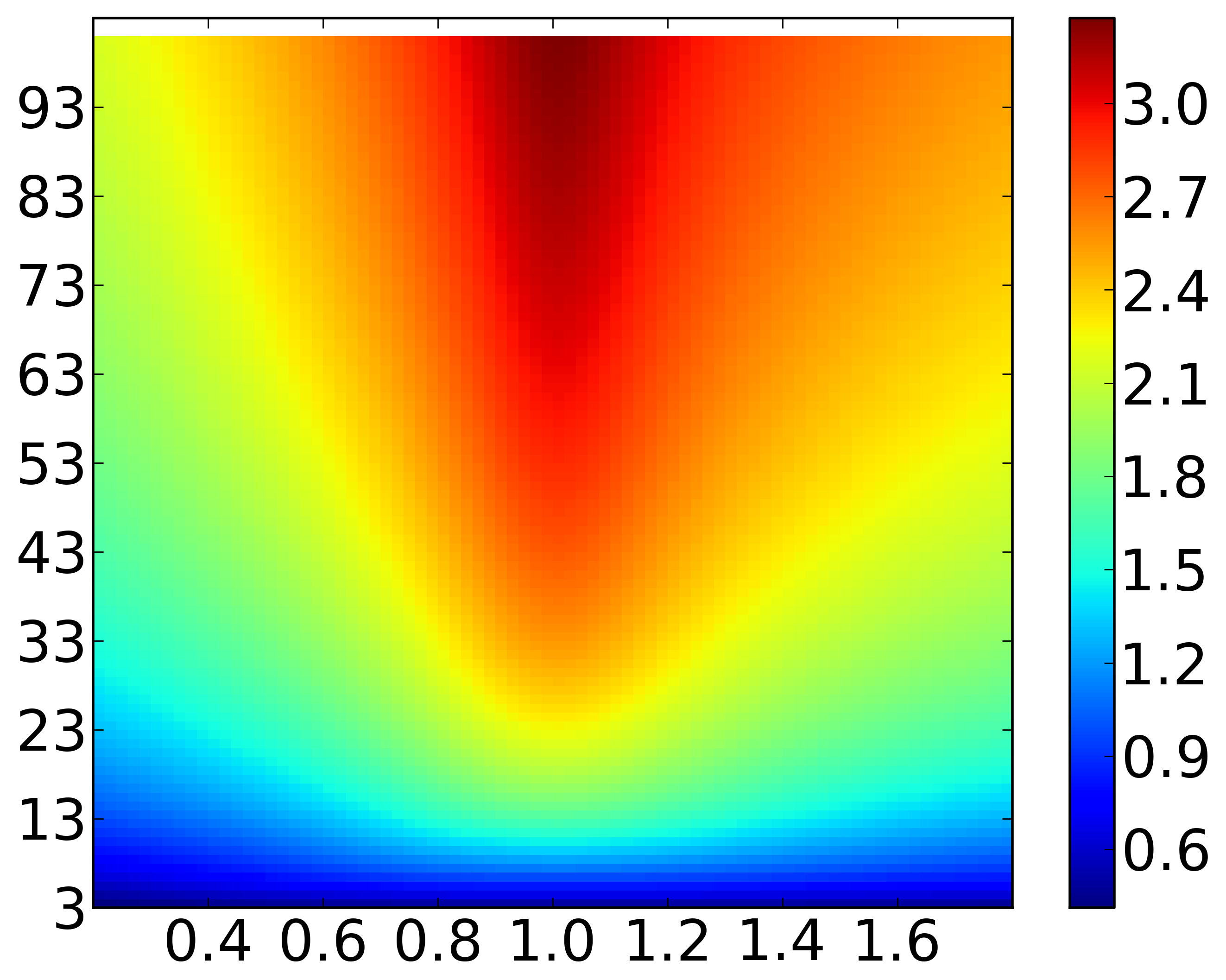}
            \caption{$\mu=0.04$}
%             \label{}
        \end{subfigure}
        \begin{subfigure}[b]{0.4\textwidth}
            \centering
            \includegraphics[width=\textwidth]{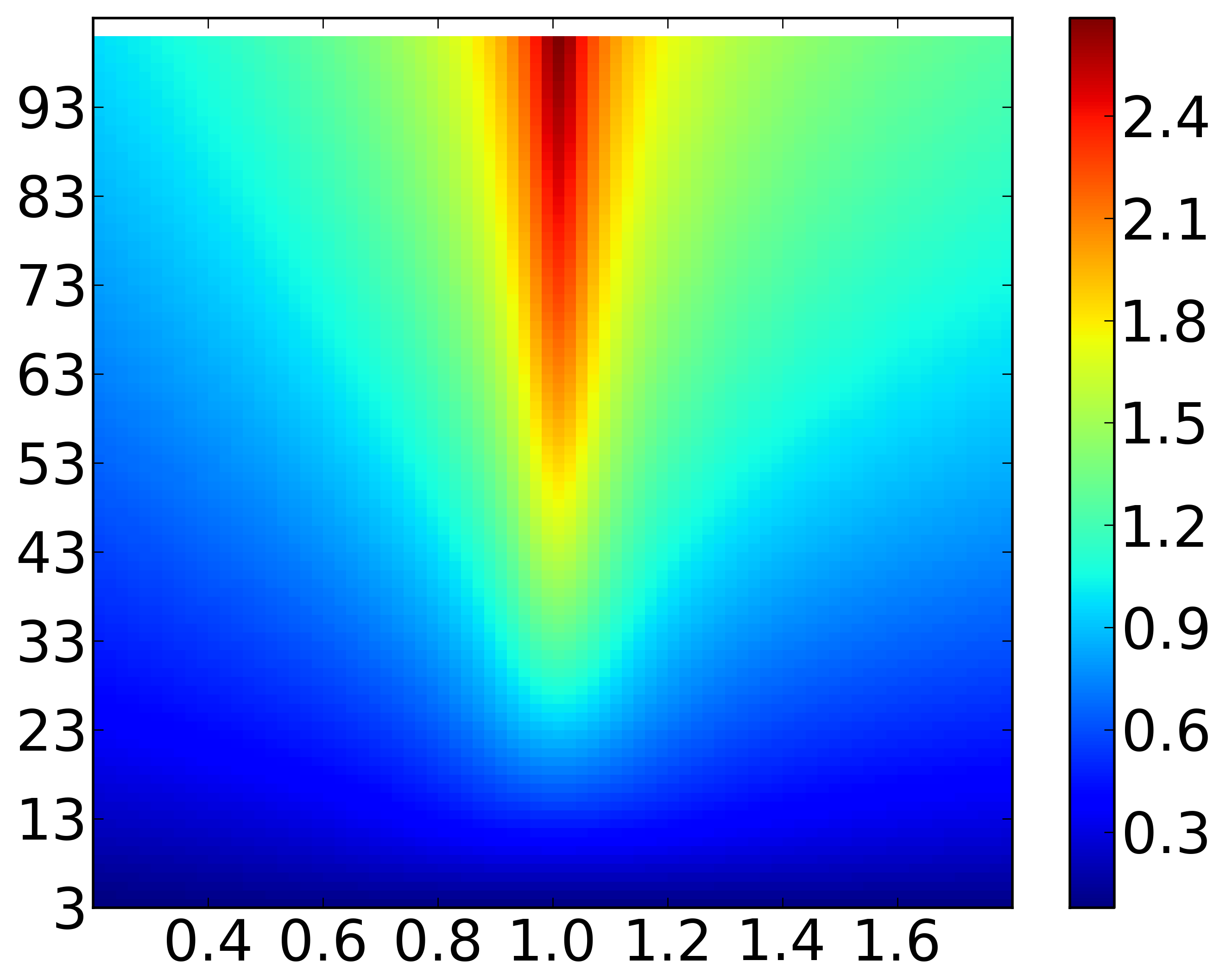}
            \caption{$\mu=0.005$}
%             \label{}
        \end{subfigure}        
        \caption{Entropy rate heatmaps for the $N$-fold Moran process, $r \in [0.4, 1.6]$ in the horizontal axis, $N \in [2, 100]$ on the vertical axis. Top row: Boundary mutation regime. Bottom Row: Uniform mutation regime. Colorbars are not consistent across plots (for additional resolution). Entropy rates are generally smaller as $\mu$ decreases (left to right). For the top row, the entropy rate is eventually decreasing as $N$ increases; for the bottom row, the entropy rate increases as $N$ increases. Compare to Figures \ref{figure_moran_constant_fitness} and \ref{figure_wf_constant_fitness}.}
        \label{figure_n_fold_constant_fitness}
\end{figure} 
\end{landscape}

\section{Other Common Game Matrices}

There are three standard fitness landscapes generated by 2x2 game matrices \cite{cressman2003evolutionary} \cite{hofbauer2003evolutionary}. See \cite{claussen2005non} for discussions of the stationary distribution for each case. Let us consider a game commonly referred to as the Hawk-Dove or Anti-coordination game, which has an interior evolutionarily stable state for continuous dynamics. The game is given by the game matrix $a=d<b=c$, and in the case where the game matrix is $a=0=d, b=1=c$ and for both mutation regimes the stationary distribution is given by 
\begin{align*}
 s_0 = s_N &= \frac{1}{2 + 2 \mu (2^N - 2)}\\
 s_j &= \frac{2 \mu}{2 + 2 \mu (2^N - 2)} \binom{N}{j}
\end{align*}
which gives another example of a nontrivial connection between $N$ and $\mu$. For this process, the denominator has a term with $\mu 2^N$ rather than $\mu N \log N$ for the neutral landscape. As such, this process is much more robust for even moderate $N$ because the factor of $2^N$ dominates the behavior of the boundary states. In other words, the stationary distribution stabilizes as $N$ increases very quickly as compared to the neutral landscape because the $2^N$ term dwarfs the other parameter contributions. Hence in general the interaction between $N$ and $\mu$ differs significantly depending on the fitness landscape. This means that taking a limit $(\mu, N) \to (0, \infty)$ to eliminate the effect of both drift and mutation is determined by the functional relationship between $\mu$ and $N$. Moreover, there may be significant implications for the commonly applied assumptions to methods like molecular clocks and the neutral theory of evolution -- the rate of evolution depends on selection, mutation, and drift. For $\mu=1/2$, this is the same distribution as in equation \ref{stationary_mu_one_half} for which the entropy rate is as given above. In particular, this shows that even for situations that would be regarded as \emph{evolutionarily stable}, there can still be a significant amount of variation in the long-run behavior of the process in a finite population, and this phenomenon is captured by the entropy rate. Moreover, despite the large mutation rate $\mu=1/2$, the stationary distribution is still strongly centered on the center state (the distribution is binomial \cite{claussen2005non} and has standard deviation less than 10 for $N=100$, indicating that conventional wisdom regarding evolutionary instability due to large mutation rates is not universal. The dependence of the stationary distribution on mutation rate is dominated by the exponential dependence on population size.

Another common game type is the prisoner's dilemma. Claussen and Traulsen consider such game having a Nash equilibrium at the state $i=0$, defined by $a=3, b=0, c=5, d=1$ \cite{claussen2005non}. They also note that $T_{1 \to 2} = 0$ if self-interaction is not allowed, and assume a small mutation rate from state 1 to 2, which has to be done for the boundary mutation regime to have a stationary distribution. For the uniform regime, $T_{1 \to 2} \neq 0$ and there is a unique stationary state. This example illustrates the classic mutation-selection balance. For small mutation probabilities, the stationary distribution concentrates at $i=0$, but for larger rates the stationary distribution moves away from the state $i=0$, with both types ``surviving''. This is also a case in which the uniform regime is both mathematically and realistically superior.

\section{Discussion}

The entropy rate of both the Moran process and the Wright-Fisher process are bounded and less than the theoretical maximum entropy rate attainable by a Markov process on $N+1$ states, where $N$ is the population size. The bound for the Wright-Fisher process depends on the population size and is approximately one-half the maximum theoretical value for large populations. For the Moran process, the bound is independent of the population size and is a much larger fraction of the theoretical maximum (94.6\%). These results imply that the inherent randomness of evolutionary processes, in so far as they are modeled by the Moran process and the Wright-Fisher process, are fundamentally bounded for fixed population sizes. Moreover, as the proof in the appendix shows, the bound for the Moran process is a consequence of fitness proportionate selection. This means that either evolutionary processes are fundamentally ordered to some extent or that the processes considered here are not accurate models of evolutionary processes. Given the many applications of these models \cite{david9stochastic} \cite{traulsen2009stochastic} \cite{ohtsuki2007one}, this work concludes that there is both order and randomness in these processes that is characterized at least in part by the entropy rate.

Intuitively, for all the processes considered, the entropy rate is maximal for neutral fitness landscapes; nevertheless, the highest entropy rate can occur for large populations when the other parameters are held constant, indicating sources of randomness in long-run population behavior can overcome those due to small population size. Though there are multiple approaches to mutation in the literature for the Moran process, the two approaches prominently discussed in this manuscript share the property that the entropy rate tends to zero as the mutation rate tends to zero. Mutation adds diversity to evolving populations, so it is also intuitive that the inherent randomness of population states is strongly dependent on mutation rates. Nevertheless, there can be very different relationships between the mutation rate and population size. As we have seen in cases where explicit calculations are possible, these parameters can interact directly or in a more complex manner, with the population size dominating the behavior in some cases and the mutation rate in others. Both interact significantly with the fitness landscape.

Finally, let us consider the meaning of \emph{inherent randomness} as an interpretation of entropy rate in this context. Processes with entropy rate approaching zero are those that fixate and occupy few states with significant probability. Relatively large entropy rates can occur from flat landscapes and well-spread stationary distributions, or tight coupling between the states with high transition entropy and stationary distribution occupation. Hence randomness can be the result of movement between many population states and more frequent movement between a smaller number of states. The distributions around the most stable states \cite{claussen2005non}, in terms of the stability theory of evolutionary games, can have a significant impact on the entropy rate. Typically the entropy rates for the Wright-Fisher and $N$-fold Moran process are similar, much closer in value, and larger, than the entropy rate of the Moran process, which is an intuitively ``less random'' process, consisting of many incremental shifts rather than generational sampling.

\subsection*{Methods}
All computations were performed with python code available at \url{https://github.com/marcharper/entropy_rate}. All plots created with \emph{matplotlib} \cite{Hunter:2007}.

\subsection*{Acknowledgments}

This research was supported by the Office of Science (BER), U. S. Department of Energy, Cooperative Agreement No. DE-FC02-02ER63421. The author thanks Christopher Strelioff for useful discussion on earlier versions of this manuscript.

\section{Appendix}

\begin{proof}[Proof of Theorem 1]
The entropy rate can be written as $E = s_0 H((\mu, 1-\mu)) + s_N H((k \mu, 1-k\mu)) + \sum_{i=1}^{N-1}{ s_i{H(T_i)}}$. As $\mu \to 0$, the first two terms converge to zero since $H((0,1)) = 0$ and $s_0 + s_N \to 1$, and the sum converges to zero since $s_i \to 0$ for $i \neq 0, n$. The latter holds because the transition probabilities depend at most linearly on $\mu$, and so $H(T_i)$ cannot prevent $s_i H(T_i)$ from converging to zero for $0 < i < N$ as $\mu \to 0$.
\end{proof}

\begin{proof}[Proof of Theorem 2]
The proof is essentially the same as for Theorem 1, except now we can argue that for $i \neq 0, N$, the stationary probabilities $s_i$ (Equation \ref{s_j}) have an additional factor of $\mu$ versus $s_0$ and $s_N$ because $T_{1 \to 0} = \mu$ and $T_{N \to N-1} = k \mu$. Hence as before, $s_0 + s_N \to 1$ and $s_1 + \ldots s_{N-1} \to 0$ as $\mu \to 0$. Equations \ref{s_j} and \ref{s_0} imply the same fixation probabilities as the uniform mutation case \cite{fudenberg2004stochastic}.
\end{proof}

To prove the maximum entropy rate for the Moran process, we first start with a generalization. Replace $i f_A(i)$ by $\incentive_A (i)$ and $(N-i) f_B(i)$ by $\incentive_B (i)$ to get the incentive dynamic in a finite population \cite{fryer2012existence}:

\begin{align}\label{incentive_process}
T_{i \to i+1} &= \frac{\incentive_A(i) (1 - \mu_{AB}(i)) + \incentive_B(i) \mu_{BA}(N-i)}{\incentive_A + \incentive_B} \frac{N-i}{N} \notag \\
T_{i \to i-1} &= \frac{\incentive_A(i) \mu_{AB}(i) + \incentive_B(i) (1 - \mu_{BA}(N-i))}{\incentive_A + \incentive_B} \frac{i}{N} \\
T_{i \to i} &= 1 - T_{i \to i+1} - T_{i \to i-1} \notag
\end{align}
For particular choices of incentive function, one can replace the replicator incentive with that corresponding to another evolutionary dynamic, such as the incentives for the best reply, logit, Fermi, or other incentive.

\begin{theorem}
For the incentive dynamics process defined above, the maximum entropy rate is \ds{\frac{3}{2} \log 2}. 
\end{theorem}
\begin{proof}
For the boundary mutation regime, $T_{i \to i+1}$ and $T_{i \to i-1}$ are the result of multiplying two probability distributions component-wise, namely $(\frac{\incentive_A}{\incentive_A + \incentive_B}, \frac{\incentive_B}{\incentive_A + \incentive_B})$ and $(\frac{N-i}{N}, \frac{i}{N})$. Because of this internal relationship, the entropy rate is bounded lower than the theoretical maximum of $\log 3$. To see this, consider more generally the first two terms of the Shannon entropy resulting from the component-wise product of two distributions $(x, 1-x)$ and $(y, 1-y)$:
\[ E_0 = xy \log{xy} + (1-x)(1-y) \log{(1-x)(1-y)}.\]
$E_0$ is maximal when $x = 1/2 = y$, but more generally maximal when $x = y$ for all $0 \leq x + y = c \leq 1$. Combining this with the third term of the entropy corresponding to $1 - x^2 - (1-x)^2 = 2x(1-x)$, gives
\[ E = x^2 \log{x^2} + (1-x)^2 \log{(1-x)^2} + 2x(1-x) \log 2x(1-x),\]
which has a maximum of $3/2 \log 2$ when $x = 1/2$. This corresponds to the distribution $(1/4, 1/4, 1/2)$ as seen earlier in the text, and bounds the entropy rate of \ref{incentive_process}. The same argument applies to an arbitrary mutation regime: the transitions are the product of the distributions $(\frac{N-i}{N}, \frac{i}{N})$ and 
$ \left( \frac{\incentive_A(i) (1 - \mu_{AB}(i)) + \incentive_B(i) \mu_{BA}(N-i)}{\incentive_A + \incentive_B}, \frac{\incentive_A(i) \mu_{AB}(i) + \incentive_B(i) (1 - \mu_{BA}(N-i))}{\incentive_A + \incentive_B}\right)$.
\end{proof}

\bibliography{ref}
\bibliographystyle{unsrt}

\end{document}